\documentclass[10pt]{article}

\usepackage{latexsym, amssymb, amsmath, amsthm, amscd}
\usepackage[all,cmtip]{xy}
\usepackage{amsmath}
\usepackage{amsthm}
\usepackage{amssymb}
\usepackage{amsfonts}
\usepackage{amsrefs}
\usepackage{color,authblk}
\usepackage{tikz}
\usepackage{amscd} 
\usepackage{comment}
\usepackage{mathrsfs}
\usepackage{comment}
\usepackage{comment}
\usepackage[all]{xy}
\usepackage{graphicx}
\usepackage{authblk}
\usepackage{hyperref}
\usepackage{fullpage}
\usepackage[all,cmtip]{xy}
\usepackage{relsize}
\usepackage{amsmath,amscd}
\usepackage{tikz-cd}
\usepackage{tikz}
\usetikzlibrary{matrix}
\usepackage[mathcal]{euscript}
\usepackage{txfonts}

\usepackage{hyperref}

\numberwithin{equation}{section} \DeclareMathSizes{2}{10}{12}{13}
\makeatletter
\newcommand*{\doublerightarrow}[2]{\mathrel{
  \settowidth{\@tempdima}{$\scriptstyle#1$}
  \settowidth{\@tempdimb}{$\scriptstyle#2$}
  \ifdim\@tempdimb>\@tempdima \@tempdima=\@tempdimb\fi
  \mathop{\vcenter{
    \offinterlineskip\ialign{\hbox to\dimexpr\@tempdima+1em{##}\cr
    \rightarrowfill\cr\noalign{\kern.5ex}
    \rightarrowfill\cr}}}\limits^{\!#1}_{\!#2}}}
\newcommand*{\triplerightarrow}[1]{\mathrel{
  \settowidth{\@tempdima}{$\scriptstyle#1$}
  \mathop{\vcenter{
    \offinterlineskip\ialign{\hbox to\dimexpr\@tempdima+1em{##}\cr
    \rightarrowfill\cr\noalign{\kern.5ex}
    \rightarrowfill\cr\noalign{\kern.5ex}
    \rightarrowfill\cr}}}\limits^{\!#1}}}
\makeatother

\newtheorem{thm}{Proposition}[section]
\newtheorem{Thm}[thm]{Theorem}
\newtheorem{rem}[thm]{Remark}

\newtheorem{cor}[thm]{Corollary}

\newtheorem{lem}[thm]{Lemma}
\newtheorem{defn}[thm]{Definition}

\title{Coalgebra measurings, cyclic theory and homologies of matrix algebras}

\author{}

\author{Abhishek Banerjee \footnote{Department of Mathematics, Indian Institute of Science, Bangalore, India. Email: abhishekbanerjee1313@gmail.com} $\qquad\qquad$ Surjeet Kour \footnote{Department of Mathematics, Indian Institute of Technology, Delhi, India. Email: koursurjeet@gmail.com}}

\date{}

\begin{document}

\maketitle 

\medskip

\begin{abstract}
In this paper, we consider coalgebra measurings and the maps induced by them between Hochschild and cyclic homology of algebras.  We  show that these induced maps  are well behaved with respect to the various structures appearing on Hochschild and cyclic homology, such as the $\lambda$-decomposition, the product structure, as well as the module structure of Hochschild homology over cyclic homology. Thereafter, we relate the maps between homology theories of algebras induced by a coalgebra measuring to those induced between homologies of matrix algebras. This is done in the following contexts: (a) cyclic homology and the primitive part of Lie algebra homology of the matrix algebra, (b) Hochschild homology and the primitive part of  Leibniz homology of the matrix algebra, and (c) Dihedral homology of an involutive algebra and the primitive part of Lie algebra homology of symplectic or skew symmetric matrices.
\end{abstract}

\medskip

{MSC (2020) Subject classification: 16T15, 16E40}


\section{Introduction} 

The classical notion of a coalgebra measuring (see Sweedler \cite{Sweed}) is that of a generalized map between  algebras. More precisely, let $K$ be a field and let $A$, $A'$ be unital $K$-algebras. Then, a coalgebra measuring $(C,\Phi)$ from $A$ to $A'$ consists of a $K$-coalgebra $C$ and a $K$-linear map $\Phi:C\longrightarrow Hom_K(A,A')$ that satisfies 
\begin{equation}\label{ms1.1}
\Phi(x)(ab)=\sum \Phi(x_{(1)})(a)\Phi(x_{(2)})(b)\qquad \Phi(x)(1_A)=\epsilon_C(x)1_{A'} \qquad \forall\textrm{ }a,b\in A,x\in C
\end{equation} In \eqref{ms1.1}, the coproduct $\Delta_C$ on $C$ is written as $\Delta_C(x)=\sum x_{(1)}\otimes x_{(2)}$ for each $x\in C$ and $\epsilon_C:C\longrightarrow K$ denotes the counit on $C$. If $x\in C$ is a group like element, i.e., $\Delta_C(x)=x\otimes x$ and $\epsilon_C(x)=1$, then $\Phi(x):A\longrightarrow A'$ recovers the usual notion of a $K$-algebra morphism from $A$ to $A'$. For $K$-algebras $A$, $A'$, there is also a universal measuring coalgebra $\mathcal M(A,A')$, often called the Sweedler Hom. The generalized Hom objects $\mathcal M(A,A')$ lead to an enrichment of $K$-algebras over $K$-coalgebras. The Sweedler Hom is also closely related to the notion of finite dual of an algebra, which gives an adjunction between algebras and coalgebras (see Porst and Street \cite{PS}). Accordingly, coalgebra measurings have been studied in a number of contexts, such as with differential graded algebras (see Anel and Joyal 
\cite{AJ}), with bialgebras (see Grunenfelder and Mastnak \cite{GM1}, \cite{GM2}), with entwining structures (see Brzezi\'{n}ski \cite{Brz}) as also with monoids in braided monoidal categories (see Hyland, L\'{o}pez Franco and Vasilakopoulou \cite{MLV}, Vasilakopoulou \cite{Vas}).  There is a similar notion of comodule measuring, which behaves like a generalized map between modules. For more on coalgebra measurings, we refer the reader, for instance, to \cite{Bat0}, \cite{Bat}, \cite{Vas1}, \cite{Lauve}. 

\smallskip
Since coalgebra measurings are seen as generalized morphisms between algebras, we showed in previous work in \cite{BK1}, \cite{BK2} that they can be used to construct maps between a wide range of (co)homology theories. For instance, if $\Phi:C\longrightarrow Hom_K(A,A')$ is a measuring by means of a cocommutative coalgebra $C$, there are induced maps
\begin{equation}\label{Ims1.2}
HH_\bullet^\Phi(x):HH_\bullet(A)\longrightarrow HH_\bullet(A')\qquad x\in C
\end{equation}
between Hochschild homology groups. Further, if the algebras $A$, $A'$ are commutative, the maps in \eqref{Ims1.2} induce a measuring between Hochschild homology rings 
$HH_\bullet(A)$ and $HH_\bullet(A')$ with respect to the shuffle product structures. Similarly, suppose that we have a coalgebra measuring between Lie algebras $(\mathfrak g,[\_\_,
\_\_])$ and $(\mathfrak g',[\_\_,\_\_]')$ 
\begin{equation}\label{Ims1.3}
\Psi: C\longrightarrow Hom_K(\mathfrak g,\mathfrak g') \qquad \Psi(x)([\alpha,\beta])=[\Psi(x_{(1)})(\alpha),\Psi(x_{(2)})(\beta)]'\qquad \alpha,\beta\in \mathfrak g, x\in C
\end{equation}
with $C$ cocommutative. Then, we showed in \cite{BK1} that the measuring $\Psi: C\longrightarrow Hom_K(\mathfrak g,\mathfrak g')$  induces maps between Chevalley-Eilenberg homologies of the Lie algebras $\mathfrak g$ and $\mathfrak g'$. In \cite{BK2}, we showed how coalgebra measurings induce maps between (co)homology theories for Hopf algebroids.

\smallskip
We have multiple aims in this paper. The first is to show that the maps $HH_\bullet^\Phi(x)$  in \eqref{Ims1.2} between Hochschild homologies, as previously defined in \cite{BK1}, extend to maps 
$HC_\bullet^\Phi(x): HC_\bullet(A)\longrightarrow HC_\bullet(A')$ between cyclic homology groups. We then show that the maps $HH_\bullet^\Phi(x)$  and $HC^\Phi_\bullet(x)$ are well behaved with respect to the various structures appearing on Hochschild and cyclic homology, such as the $\lambda$-decomposition, the product structure, as well as the module structure of Hochschild homology over cyclic homology.

\smallskip
Our second aim is to relate maps between homology theories of algebras induced by a coalgebra measuring to those induced between homology theories of Lie algebras or Leibniz algebras. In particular, we show that maps induced by measurings are compatible with the isomorphisms that identify cyclic homology with the primitive part of the Lie algebra homology of matrices. Similarly, we show that maps induced by measurings are compatible with the isomorphisms that identify Hochschild homology with the primitive part of the Leibniz algebra homology of matrices. Our final aim in this paper is to study coalgebra measurings between involutive algebras, in which case the cyclic homology is usually replaced with the dihedral homology. Accordingly, we show that maps induced by measurings are compatible with the isomorphisms that identify dihedral homology with the primitive part of the Lie algebra homology of symplectic matrices or that of skew symmetric matrices.

\smallskip
We mention that our motivation for studying maps between (co)homology theories induced by coalgebra measurings comes in part from an analogy with algebraic geometry. In algebraic geometry, one enlarges the category of varieties by adding generalized maps known as correspondences (see, for instance, \cite{WMV}). The correspondences allow us to capture at an algebraic level a much larger class of morphisms between (co)homology theories. These are then used to define the ``motive'' of an algebraic variety. Similarly, we note that coalgebra measurings linearize the category of algebras by adding generalized morphisms. Since an algebra may be seen as a ``noncommutative space,'' we were led to ask how coalgebra measurings capture morphisms between homology theories of algebras. 

\smallskip
We now describe the paper in more detail. We begin in Section 2 by showing that a coalgebra measuring $\Phi:C\longrightarrow Hom_K(A,A')$ between algebras $A$, $A'$, with $C$ cocommutative, induces maps  between cyclic modules. Accordingly, the first main result of this paper is as follows.

\begin{Thm}\label{hmT1.1} (see \ref{P1.1} and \ref{P2.2}) Let $C$ be a cocommutative $K$-coalgebra and let $\Phi:C\longrightarrow Hom_K(A,A')$ be a  measuring between algebras. For each $x\in C$,  we have induced maps
\begin{equation}\label{1.4ctr}
HH^\Phi_\bullet(x): HH_\bullet(A)\longrightarrow HH_\bullet(A')\qquad HC^\Phi_\bullet(x): HC_\bullet(A)\longrightarrow HC_\bullet(A')
\end{equation}
on Hochschild homology and cyclic homology groups respectively.  Further, the maps $HH^\Phi_\bullet(x)$ and 
$HC^\Phi_\bullet(x)$ in \eqref{1.4ctr} are compatible with the operators appearing in the  Connes periodicity sequences for the algebras $A$ and $A'$.

\end{Thm} Thereafter, we show that the induced maps  $HH^\Phi_\bullet(x)$ and 
$HC^\Phi_\bullet(x)$ in Theorem \ref{hmT1.1} are compatible with the operators connecting the Hochschild and cyclic complexes to the module
of K$\ddot{\mbox{a}}$hler differentials, when the algebras involved are commutative. For Hochschild homology, this refers to the anti-symmetrization and projection operators (see Proposition \ref{P2.5gx})
\begin{equation}
\varepsilon_\bullet:\Omega_A^\bullet\longrightarrow HH_\bullet(A)\qquad \pi_\bullet: HH_\bullet(A)\longrightarrow \Omega_A^\bullet
\end{equation} For the cyclic homology, we show that the maps $HC_\bullet^\Phi(x)$ for each $x\in C$ are compatible with respect to the projections
\begin{equation}
HC_n(A)\xrightarrow{\qquad\bar{\pi}_n\qquad}\Omega^n_A/d\Omega^{n-1}_A\oplus HDR_A^{n-2}\oplus HDR^{n-4}_A\oplus \dots \qquad n\geq 0
\end{equation} where $HDR^\bullet_A$ is the de Rham cohomology obtained from the complex of K$\ddot{\mbox{a}}$hler differentials (see Proposition \ref{P2.6t}).

\smallskip
We come to the product structure on cyclic homology in Section 3. If $A$ is commutative, we know (see \cite[$\S$ 3]{LQ}) that there is a product 
\begin{equation}\label{1.5ctr}
\_\_\ast\_\_:  HC_{p}(A)\otimes HC_q(A)\longrightarrow HC_{p+q+1}(A)\qquad p,q\geq 0
\end{equation} If we shift the degree by setting $HC'_\bullet(A):=HC_{\bullet-1}(A)$, we know  that the product in \eqref{1.5ctr} makes cyclic homology into an associative graded algebra $(HC'_\bullet(A),\ast)$. Further, we know that the Connes' operator $B:HC'_\bullet(A)=HC_{\bullet-1}(A)\longrightarrow HH_\bullet(A)$ is a map of graded algebras, making the Hochschild homology into a module over the cyclic homology. The following result shows how a coalgebra measuring leads to a measuring between cyclic homology rings, as well as a comodule measuring at the level of Hochschild homology.

\begin{Thm}\label{hmT1.2} (see \ref{T3.2} and \ref{P3.3st}) Let $C$ be a cocommutative coalgebra and let $\Phi:C\longrightarrow Hom_K(A,A')$ be a measuring between commutative algebras. Then, we have

\smallskip
(a) The maps
$
HC'^\Phi_\bullet(x):HC'_\bullet(A)=HC_{\bullet-1}(A)\xrightarrow{ HC_{\bullet-1}^\Phi(x)}HC_{\bullet-1}(A')=HC'_{\bullet}(A')$, $x\in C$ 
 determine a coalgebra measuring
of algebras from $(HC'_\bullet(A),\ast)$ to $(HC'_\bullet(A'),\ast)$.

\smallskip
(b) The maps $HH_\bullet^\Phi(x):HH_\bullet(A)\longrightarrow HH_\bullet(A')$, $x\in C$ 
  determine a comodule measuring  from the $(HC'_\bullet(A),\ast)$-module $HH_\bullet(A)$ to the $(HC'_\bullet(A'),\ast)$-module $HH_\bullet(A')$, where $C$ is treated as a $C$-comodule.
\end{Thm}

Using Theorem \ref{hmT1.2}, we construct an enrichment $c\widetilde{ALG}_K$ of commutative $K$-algebras over the symmetric monoidal category 
$coCoalg_K$ of cocommutative $K$-coalgebras. For commutative $K$-algebras $A$, $A'$, the Hom object in $c\widetilde{ALG}_K$ is taken to be 
\begin{equation} c\widetilde{ALG}_K(A,A'):=\mathcal M_c((HC'_\bullet(A),\ast),(HC'_\bullet(A'),\ast))
\end{equation} where $\mathcal M_c(\_\_,\_\_)$ denotes the universal cocommutative
measuring coalgebra. Further, we show that there is a $coCoalg_K$-enriched functor $cALG_K\longrightarrow c\widetilde{ALG}_K$, where $cALG_K$ is the enrichment of commutative $K$-algebras obtained by taking $cALG_K(A,A'):=\mathcal M_c(A,A')$ (see Theorem \ref{T3.6}). 

\smallskip
 For a commutative $K$-algebra $A$, we now recall the $\lambda$-decompositions (see, for instance, \cite[$\S$ 4]{Lod}) for Hochschild and cyclic homology
\begin{equation} \label{1/10ctr}
HH_n(A)=HH_n^{(1)}(A)\oplus ... \oplus HH_n^{(n)}(A)\qquad HC_n(A)=HC_n^{(1)}(A)\oplus ... \oplus HC_n^{(n)}(A)  
\end{equation}  Our objective in Section 4  is to show that the maps induced by coalgebra measurings are compatible with the direct sum decompositions in \eqref{1/10ctr}. In this respect, our main result is as follows.

\begin{Thm}\label{hmT1.3} (see \ref{T4.4}, \ref{T4.6}, \ref{C4.7})
Let $C$ be a cocommutative $K$-coalgebra and let $\Phi:C\longrightarrow Hom_K(A,A')$ be a  measuring between commutative algebras. Then, for  each $x\in C$ and $i,n\geq 0$, we have induced maps 
\begin{equation}
HH^{\Phi,(i)}_\bullet(x): HH_\bullet^{(i)}(A)\longrightarrow HH_\bullet^{(i)}(A') \qquad HC^{\Phi,(i)}_\bullet(x): HC_\bullet^{(i)}(A)\longrightarrow HC_\bullet^{(i)}(A')
\end{equation}
on direct summands appearing in the $\lambda$-decomposition of Hochschild and cyclic homology groups. Further, these maps are compatible with the operators appearing in the  Connes periodicity sequences for the algebras $A$ and $A'$.  
\end{Thm}

In Section 5, we begin relating the maps on Hochschild and cyclic homology induced by a coalgebra measuring to those induced on Lie algebra homology or Leibniz homology of matrix algebras. As mentioned before, we have shown in \cite{BK1} that a coalgebra measuring $\Psi: C\longrightarrow Hom_K(\mathfrak g,\mathfrak g')$ between Lie algebras in the sense of \eqref{Ims1.3} induces a map $H^\Psi_\bullet(x):H_\bullet(\mathfrak g)\longrightarrow H_\bullet(\mathfrak g')$ between Chevalley-Eilenberg homologies for each  $x\in C$. Using the Leibniz homology introduced by Loday \cite[$\S$ 6.3]{Lod3}, we show that there are similar maps $HL^\Psi_\bullet(x):HL_\bullet(\mathfrak g)\longrightarrow HL_\bullet(\mathfrak g')$ for a coalgebra measuring $\Psi: C\longrightarrow Hom_K(\mathfrak g,\mathfrak g')$ between Leibniz algebras $\mathfrak g$ and $\mathfrak g'$. If $\mathfrak g$ is a Lie algebra (resp. a Leibniz algebra) we know that the Chevalley-Eilenberg homology $H_\bullet(\mathfrak g)$ (resp. the Leibniz homology $HL_\bullet(\mathfrak g)$) is equipped with a coproduct $\Delta_{\mathfrak g}: H_\bullet(\mathfrak g)
\longrightarrow H_\bullet(\mathfrak g)\otimes H_\bullet(\mathfrak g)$ (resp. a coproduct  $\Delta_{\mathfrak g}: HL_\bullet(\mathfrak g)
\longrightarrow HL_\bullet(\mathfrak g)\boxtimes HL_\bullet(\mathfrak g)$) induced by the diagonal embedding $\mathfrak g\longrightarrow 
\mathfrak g\oplus \mathfrak g$. We show that these coproducts are well behaved with respect to a coalgebra measuring  $\Psi: C\longrightarrow Hom_K(\mathfrak g,\mathfrak g')$, i.e., we have commutative diagrams (see Proposition \ref{P5.1} and Proposition \ref{P5.5}) for each $x\in C$
\begin{equation}
\begin{array}{ccc}
\begin{CD}
H_\bullet(\mathfrak g) @>\Delta_{\mathfrak g}>> H_\bullet(\mathfrak g) \otimes H_\bullet(\mathfrak g) \\
@VH^\Psi_\bullet(x)VV @VVH^\Psi_\bullet(x_{(1)})\otimes H^\Psi_\bullet(x_{(2)})V\\
H_\bullet(\mathfrak g') @>\Delta_{\mathfrak g'}>> H_\bullet(\mathfrak g') \otimes H_\bullet(\mathfrak g') \\
\end{CD} &\qquad \qquad & \begin{CD}
HL_\bullet(\mathfrak g) @>\Delta_{\mathfrak g}>> HL_\bullet(\mathfrak g) \boxtimes HL_\bullet(\mathfrak g) \\
@VHL^\Psi_\bullet(x)VV @VVHL^\Psi_\bullet(x_{(1)})\boxtimes HL^\Psi_\bullet(x_{(2)})V\\
HL_\bullet(\mathfrak g') @>\Delta_{\mathfrak g'}>> HL_\bullet(\mathfrak g') \boxtimes HL_\bullet(\mathfrak g') \\
\end{CD} \\
\end{array}
\end{equation} We let $gl(A)$ denote the Lie algebra of matrices over $A$, which may also be treated as a Leibniz algebra. We know (see  \cite[$\S$ 10.2]{Lod} and  \cite[Th\'{e}or\`{e}me III.3]{Cuv1}) that $H_\bullet(gl(A))$ and $HL_\bullet(gl(A))$ carry  Hopf algebra structures and their respective primitive parts are related by isomorphisms
\begin{equation}\label{primis1.1}
Prim(H_\bullet(gl(A)))\cong HC_{\bullet-1}(A) \qquad Prim(HL_\bullet(gl(A)))\cong HH_{\bullet-1}(A)
\end{equation}
to cyclic and Hochschild homologies. If $C$ is cocommutative, a coalgebra measuring $\Phi:C\longrightarrow Hom_K(A,A')$ between algebras induces a coalgebra measuring  $gl(\Phi):C\longrightarrow Hom_K(gl(A),gl(A'))$ between Lie algebras, as also at the level of Leibniz algebras. The main result of Section 5 is to show that the isomorphisms in \eqref{primis1.1} are compatible with the maps induced by coalgebra measurings.

\begin{Thm}\label{hmT1.5x} (see \ref{T5.4} and \ref{T5.7v})
Let $C$ be a cocommutative $K$-coalgebra and let $\Phi:C\longrightarrow Hom_K(A,A')$ be a  measuring between  algebras. For each $x\in C$, the following diagram commutes
\begin{equation}
\begin{array}{ccc}\begin{CD}
Prim(H_\bullet(gl(A)))@>\cong>> HC_{\bullet-1}(A)\\
@VH_\bullet^{gl(\Phi)}(x)VV @VVHC_{\bullet-1}^\Phi(x) V \\
Prim(H_\bullet(gl(A')))@>\cong>> HC_{\bullet-1}(A')\\
\end{CD}&\qquad\qquad &
\begin{CD}
Prim(HL_\bullet(gl(A)))@>\cong>> HH_{\bullet-1}(A)\\
@VHL_\bullet^{gl(\Phi)}(x)VV @VVHH_{\bullet-1}^\Phi(x) V \\
Prim(HL_\bullet(gl(A')))@>\cong>> HH_{\bullet-1}(A')\\
\end{CD}
\end{array}
\end{equation} where the horizontal isomorphisms are as given by \eqref{primis1.1}.

\end{Thm}

We come to the product structure on Lie algebra homology and Leibniz homology in Section 6. We know that both $H_\bullet(gl(A))$ and $HL_\bullet(gl(A))$ are equipped with product structures induced by the operation $\oplus: gl(A)\times gl(A)\longrightarrow gl(A)$ as described in \cite[$\S$ 10.2.12]{Lod}. We then show the following result.

\begin{Thm}\label{hmT1.6y} (see \ref{P6.2} and \ref{P6.2L})  Let $C$ be a cocommutative coalgebra and let $\Phi:C\longrightarrow Hom_K(A,A')$ be a measuring between  algebras. Then, the maps 
\begin{equation} 
H_\bullet^{gl(\Phi)}:C\longrightarrow Hom_K(H_\bullet(gl(A)),H_\bullet(gl(A')))\qquad x\mapsto H_\bullet^{gl(\Phi)}(x):H_\bullet(gl(A))\longrightarrow H_\bullet(gl(A'))
\end{equation} determine a  measuring of algebras from $H_\bullet(gl(A))$ to $H_\bullet(gl(A'))$.  Similarly, 
 the maps \begin{equation}\label{mslie6Lx}
HL_\bullet^{gl(\Phi)}:C\longrightarrow Hom_K(HL_\bullet(gl(A)),HL_\bullet(gl(A')))\qquad x\mapsto HL_\bullet^{gl(\Phi)}(x):HL_\bullet(gl(A))\longrightarrow HL_\bullet(gl(A'))
\end{equation} determine a measuring of algebras from $HL_\bullet(gl(A))$ to $HL_\bullet(gl(A'))$. 
\end{Thm}

Using Theorem \ref{hmT1.6y}, we construct  enrichments $\overline{ALG}^c_K$ and $\widehat{ALG}^c_K$ of  $K$-algebras over the symmetric monoidal category 
$coCoalg_K$ of cocommutative $K$-coalgebras. For  $K$-algebras $A$, $A'$, the respective Hom objects are taken to be 
\begin{equation} \overline{ALG}^c_K(A,A'):=\mathcal M_c(H_\bullet(gl(A)),H_\bullet(gl(A')))\qquad  \widehat{ALG}^c_K(A,A'):=\mathcal M_c(HL_\bullet(gl(A)),HL_\bullet(gl(A')))
\end{equation} where $\mathcal M_c(\_\_,\_\_)$ denotes the universal cocommutative
measuring coalgebra. Accordingly, we have $coCoalg_K$-enriched functors $ALG_K^c\longrightarrow  \overline{ALG}^c_K $ and $ALG_K^c\longrightarrow    \widehat{ALG}^c_K$ , where $ALG_K^c $ is the enrichment of $K$-algebras obtained by taking $ALG_K^c(A,A'):=\mathcal M_c(A,A')$ (see Theorem \ref{T3.6cx} and Theorem \ref{T3.6cxL}).

\smallskip
In the final section of the paper, we work with involutive algebras. We recall  (see \cite[$\S$ 10.5]{Lod})  that a $K$-algebra $R$ is said to be involutive if it is equipped with a $K$-linear endomorphism $r\mapsto \hat{r}$ which satisfies
  $\hat{\hat{r}}=r$ and $
\hat{r}\hat{s}=\widehat{sr}
$ for $r$, $s\in R$. For an involutive $K$-algebra $R$, the usual counterpart of cyclic homology is given by the dihedral homology groups
$H\mathcal D_\bullet(R)$  (see \cite{Lod87}, \cite{LP88}, \cite[$\S$ 10.5]{Lod}). Accordingly, we define a coalgebra measuring 
$\Phi: C\longrightarrow Hom_K(R,R')$ between involutive algebras $R$, $R'$ and show that there are induced maps $H\mathcal D^\Phi_\bullet(x):H\mathcal D_\bullet(R)\longrightarrow H\mathcal D_\bullet(R')$ between dihedral homology groups for each $x\in C$. 

\smallskip
For an involutive $K$-algebra $R$,  let $sk(R)$ the Lie algebra of skew symmetric matrices over $R$ and let $sp(R)$ be the Lie algebra of symplectic matrices over $R$ (see Section 7 for definitions). Accordingly, we show in Proposition \ref{P7.4} that a cocommutative measuring $\Phi:C\longrightarrow Hom_K(R,R')$  between involutive algebras  induces maps
\begin{equation} 
H_\bullet^{sk(\Phi)}(x):H_\bullet(sk(R))\longrightarrow H_\bullet(sk(R'))\qquad H_\bullet^{sp(\Phi)}(x):H_\bullet(sp(R))\longrightarrow H_\bullet(sp(R'))
\end{equation}
between Lie algebra homologies for each $x\in C$. We now recall (see \cite[$\S$ 10.5]{Lod}) that  $H_\bullet(sk(R))$ and
$H_\bullet(sp(R))$ are actually Hopf algebras and there are isomorphisms
\begin{equation}\label{7.12qi}
Prim(H_\bullet(sk(R)))\overset{\cong}{\longrightarrow} H\mathcal D_{\bullet-1}(R) \qquad Prim(H_\bullet(sp(R)))\overset{\cong}{\longrightarrow} H\mathcal D_{\bullet-1}(R)
\end{equation} relating their primitive parts to the dihedral homology. 
Our final result is to show that the isomorphisms in \eqref{7.12qi} are compatible with maps induced by coalgebra measurings between involutive algebras. 

\begin{Thm}\label{hmT1.8u} (see \ref{T57.4}) 
Let $R$, $R'$ be involutive $K$-algebras. Let $\Phi:C\longrightarrow Hom_K(R,R')$ be a cocommutative measuring of involutive algebras from $R$ to $R'$.  For each $x\in C$, the following diagrams commute
\begin{equation}
\begin{array}{lll}
\begin{CD}
Prim(H_\bullet(sk(R)))@>\cong>> H\mathcal D_{\bullet-1}(R)\\
@VH_\bullet^{sk(\Phi)}(x)VV @VVH\mathcal D_{\bullet-1}^\Phi(x) V \\
Prim(H_\bullet(sk(R')))@>\cong>> H\mathcal D_{\bullet-1}(R')\\
\end{CD} & \qquad &
\begin{CD}
Prim(H_\bullet(sp(R)))@>\cong>> H\mathcal D_{\bullet-1}(R)\\
@VH_\bullet^{sp(\Phi)}(x)VV @VVH\mathcal D_{\bullet-1}^\Phi(x) V \\
Prim(H_\bullet(sp(R')))@>\cong>> H\mathcal D_{\bullet-1}(R')\\
\end{CD} \\
\end{array}
\end{equation} where the isomorphisms in the horizontal arrows are as given by \eqref{7.12qi}.
\end{Thm}

\section{Measurings and morphisms on cyclic homology}

Let $K$ be a field of characteristic $0$. Let  $A$ be a $K$-algebra. For each $p\geq 0$, we set $C_p(A):=A^{\otimes p+1}$. By abuse of notation, we will often denote an element 
$a_0\otimes a_1\otimes ...\otimes a_p\in C_p(A)$ simply as $(a_0,a_1,...,a_p)$. We know (see, for instance, \cite[$\S$ 1]{Lod}) that the simplicial module $C_\bullet(A)$ associated to $A$  consists of the terms $C_p(A)=A^{\otimes p+1}$, ${p\geq 0}$ along with face maps 
\begin{equation}
\begin{array}{ll}
d_i^p:C_p(A)\longrightarrow C_{p-1}(A) \qquad & d_i^p(a_0,a_1,...,a_p):=
\left\{
\begin{array}{ll}
(a_0,...,a_ia_{i+1},...,a_p) & \mbox{$0\leq i\leq p-1$} \\
(a_pa_0,a_1,...,a_{p-1}) & \mbox{$i=p$}\\
\end{array}\right.\\
\end{array} 
\end{equation} and degeneracy maps
\begin{equation}
s_j^p:C_p(A)\longrightarrow C_{p+1}(A)\qquad s_j^p(a_0,...,a_p):=(a_0,..,a_j,1,a_{j+1},...,a_p)
\end{equation} for $0\leq j\leq p$. The Hochschild homology groups $HH_\bullet(A)$ of $A$ are given by the homology of the standard chain complex $C^{hoc}_\bullet(A)$ associated to the simplicial module $C_\bullet(A)$ (see \cite[$\S$ 1.6.1]{Lod}). Together with the cyclic operator
\begin{equation}
t_p:C_p(A)\longrightarrow C_p(A)\qquad t_p(a_0,...,a_p)=(-1)^{p}(a_p,a_0,...,a_{p-1})
\end{equation} one has the cyclic module associated to $A$ (see \cite[$\S$ 2.5.4]{Lod}), which we continue to denote by $C_\bullet(A)$. Whenever it is clear from context, we will denote the face maps, degeneracy maps and cyclic operator  simply by $d_i:C_p(A)\longrightarrow C_{p-1}(A) $, $s_j:C_p(A)\longrightarrow C_{p+1}(A)$, and $t:C_p(A)\longrightarrow C_p(A)$ respectively.  We know that the differential on the Hochschild complex 
$(C_\bullet^{hoc}(A),b)$ is given by $b:=\sum_{i=0}^p(-1)^p d_i:C_p(A)\longrightarrow C_{p-1}(A)$. The cyclic homology
groups $HC_\bullet(A)$ of $A$ are given by the homology of the total complex $Tot(C_{\bullet\bullet}^{cy}(A))$, where $C_{\bullet\bullet}^{cy}(A)$ is  the bicomplex  associated to the cyclic module $C_\bullet(A)$ (see \cite[$\S$ 2.5.5]{Lod}).

\smallskip 
Now let $A$, $A'$ be  $K$-algebras. Let $C$ be a cocommutative $K$-coalgebra having  coproduct 
$\Delta_C:C\longrightarrow C\otimes C$ and counit $\epsilon_C:C\longrightarrow K$. Let  
$\Phi:C\longrightarrow Hom_K(A,A')$ be a measuring from $A$ to $A'$ as described in \eqref{ms1.1}. For $x\in C$, we will typically write the coproduct  on $C$ using Sweedler notation as $\Delta_C(x)=x_{(1)}\otimes 
x_{(2)}$ by suppressing summation signs. For $x\in C$ and $a\in A$, we will often write the element $\Phi(x)(a)\in A'$  simply as 
$x(a)\in A'$. For each $x\in C$, we now set
\begin{equation}\label{2.4mao}
C_n^\Phi(x): C_n(A)\longrightarrow C_n(A')\qquad (a_0,...,a_n)\mapsto (x_{(1)}(a_0),...,x_{(n+1)}(a_n))
\end{equation}
We will often write $C_n^\Phi(x)(a_0,...,a_n)$ simply as $x(a_0,...,a_n)$. We start with the following result.

\begin{thm}\label{P1.1}
Let $C$ be a cocommutative $K$-coalgebra and let $\Phi:C\longrightarrow Hom_K(A,A')$ be a  measuring between algebras. For each $x\in C$, the maps
$\{C_n^\Phi(x): C_n(A)\longrightarrow C_n(A')\}_{n\geq 0}$ together determine a morphism $C_\bullet^\Phi(x):C_\bullet(A)\longrightarrow C_\bullet(A')$ of cyclic modules. In particular, we have induced maps
\begin{equation}
HH^\Phi_\bullet(x): HH_\bullet(A)\longrightarrow HH_\bullet(A')\qquad HC^\Phi_\bullet(x): HC_\bullet(A)\longrightarrow HC_\bullet(A')
\end{equation}
on Hochschild homology and cyclic homology groups respectively. 
\end{thm}

\begin{proof} Let $x\in C$. We know from \cite[Proposition 2.2]{BK1} that the maps 
$C_n^\Phi(x): C_n(A)\longrightarrow C_n(A')$ are well behaved with respect to the face maps on the cyclic modules $C_\bullet(A)$ and $C_\bullet(A')$. For $p\geq 0$, $0\leq j\leq p$ and $(a_0,...,a_p)\in C_p(A)$, we now verify  that
\begin{equation*}
\begin{array}{ll}
C_n^\Phi(x)s_j(a_0,...,a_p)& = (x_{(1)}(a_0),...,x_{(j+1)}(a_j),x_{(j+2)}(1),x_{(j+3)}(a_{j+1}),...,x_{(p+2)}(a_p))\\
&=(x_{(1)}(a_0),...,x_{(j+1)}(a_j),\epsilon_C(x_{(j+2)})\cdot 1 ,x_{(j+3)}(a_{j+1}),...,x_{(p+2)}(a_p))\\
&=(x_{(1)}(a_0),...,x_{(j+1)}(a_j),1 ,x_{(j+2)}(a_{j+1}),...,x_{(p+1)}(a_p))\\
&=s_jC_n^\Phi(x)(a_0,...,a_p)\\
\end{array}
\end{equation*} Since $C$ is cocommutative, we also see that 
\begin{equation*}
C_n^\Phi(x)t(a_0,...,a_p)=(-1)^p(x_{(1)}(a_p),x_{(2)}(a_0),...,x_{(p+1)}(a_{p-1}))=(-1)^p(x_{(p+1)}(a_p),x_{(1)}(a_0),...,x_{(p)}(a_{p-1}))=tC_n^\Phi(x)(a_0,...,a_p)\\
\end{equation*} It follows that the maps $C_\bullet^\Phi(x)$ are also well behaved with respect to the degeneracy and cyclic operators on the cyclic modules
$C_\bullet(A)$ and $C_\bullet(A')$. The result is now clear. 

\end{proof}

For any algebra $A$, we know that its Hochschild and cyclic homology groups fit into the following long exact sequence (see, for instance, \cite[$\S$ 2.2.1]{Lod})
\begin{equation} \label{pr2.6}
\dots \longrightarrow HH_n(A)\xrightarrow{\quad I\quad } HC_n(A) \xrightarrow{\quad S\quad } HC_{n-2}(A) \xrightarrow{\quad B\quad } HH_{n-1}(A) \longrightarrow \dots
\end{equation}
where $S:HC_\bullet(A)\longrightarrow HC_{\bullet-2}(A)$ is the Connes periodicity operator.  Our next result is that the morphisms in Proposition \ref{P1.1} are well behaved with respect to Connes periodicity. 

\begin{thm}\label{P2.2}
Let $C$ be a cocommutative $K$-coalgebra and let $\Phi:C\longrightarrow Hom_K(A,A')$ be a  measuring between algebras. For each $x\in C$, the maps $HH^\Phi_\bullet(x)$ and 
$HC^\Phi_\bullet(x)$ fit into the following commutative diagram
\begin{equation}
\begin{CD}
\dots @>>> HH_n(A) @>I>> HC_n(A) @>S>> HC_{n-2}(A) @>B>> HH_{n-1}(A) @>>> \dots \\
@. @VHH^\Phi_n(x)VV @VHC^\Phi_n(x)VV @VHC^\Phi_{n-2}(x)VV @VHH^\Phi_{n-1}(x)VV @. \\
\dots @>>> HH_n(A') @>I>> HC_n(A') @>S>> HC_{n-2}(A') @>B>> HH_{n-1}(A') @>>> \dots \\
\end{CD}
\end{equation} where the top and bottom rows are the Connes periodicity sequences for the algebras $A$ and $A'$ respectively.
\end{thm} 

\begin{proof}
 By definition (see \cite[$\S$ 2.2.1]{Lod}), the periodicity sequence in \eqref{pr2.6} is the long exact sequence of homology groups associated to the short exact sequence of bicomplexes
\begin{equation}\label{2.8ex}
0\longrightarrow C^{cy}_{\bullet\bullet}(A)\{2\} \longrightarrow C^{cy}_{\bullet\bullet}(A)\longrightarrow C^{cy}_{\bullet\bullet}(A)[2,0]\longrightarrow 0
\end{equation} Here $ C^{cy}_{\bullet\bullet}(A)\{2\} $ denotes the first two columns of the bicomplex $ C^{cy}_{\bullet\bullet}(A) $, which is quasi-isomorphic to the Hochschild complex. Further, $C^{cy}_{\bullet\bullet}(A)[2,0]$ denotes the shifted bicomplex obtained by setting $(C^{cy}_{\bullet\bullet}(A)[2,0])_{pq}:=C^{cy}_{p-2,q}(A)$ and
$C^{cy}_{\bullet\bullet}(A)\longrightarrow C^{cy}_{\bullet\bullet}(A)[2,0]$ is the canonical projection. From Proposition \ref{P1.1}, we know that the measuring $\Phi:C\longrightarrow Hom_K(A,A')$  induces a morphism  $C_\bullet^\Phi(x):C_\bullet(A)\longrightarrow C_\bullet(A')$ of cyclic modules. Since the sequence in \eqref{2.8ex} is functorial in terms of cyclic modules, the result follows.  
\end{proof}

We know that the symmetric group $S_n$ has a left action on 
$C_\bullet(A)$ given by setting, for $\sigma \in S_n$ and $(a_0,...,a_n)\in C_n(A)$,
\begin{equation}\label{sym2e}
\sigma\cdot (a_0,a_1,...,a_n):=(a_0,a_{\sigma^{-1}(1)},...,a_{\sigma^{-1}(n)})
\end{equation} In other words, $C_n(A)$ becomes a left module over the group ring $K[S_n]$. For each $n$, we now set $\varepsilon_n:=\underset{\sigma\in S_n}{ \sum}sgn(\sigma)\sigma\in K[S_n]$.

\smallskip
Our next aim is to show that the maps $\{HH^\Phi_\bullet(x)\}_{x\in C}$ induced by the measuring are well behaved with respect to the anti-symmetrization maps connecting the Chevalley-Eilenberg complex of $A$ (treated as a Lie algebra by setting $[a,b]:=ab-ba$ for $a$, $b\in A$) to the Hochschild complex of $A$ (see \cite[$\S$ 1.3.4]{Lod}). For the $K$-algebra $A$, we have the Chevalley-Eilenberg complex $(E_\bullet(A),\delta)$ given by 
setting $E_n(A):=A\otimes \wedge^nA$, $n\geq 0$ along with differentials
\begin{equation}\label{2.10qz}
\begin{array}{c}
\delta: E_n(A)=A\otimes \wedge^nA \xrightarrow{\qquad\qquad\qquad} A\otimes \wedge^{n-1}A =E_{n-1}(A)\\
\begin{array}{ll}
\delta(a_0\otimes a_1\wedge ...\wedge a_n)&=\underset{i=1}{\overset{n}{\mathlarger\sum}} (-1)^i[a_0,a_i]\otimes a_1\wedge ... \wedge \hat{a}_i \wedge ... \wedge a_n\\ 
& \textrm{ }\quad+ \underset{1\leq i\leq j\leq n}{{\mathlarger\sum}} (-1)^{i+j-1} a_0\otimes [a_i,a_j]\wedge a_1\wedge ... \wedge \hat{a}_i\wedge ... \wedge \hat{a}_j\wedge ...\wedge a_n\\
\end{array} \\
\end{array}
\end{equation} where $\hat{a}_i$ in \eqref{2.10qz} refers to $a_i$ removed from the sequence. We will denote by $HE_\bullet(A)$ the homology of the complex $(E_\bullet(A),\delta)$. We recall (see \cite[$\S$ 1.3.5]{Lod}) that the anti-symmetrization map 
\begin{equation}
\varepsilon_n: E_n(A)=A\otimes \wedge^nA\longrightarrow A\otimes A^{\otimes n} \qquad a_0\otimes a_1\wedge ...\wedge a_n \mapsto 
\varepsilon_n\cdot (a_0\otimes a_1\otimes ... \otimes a_n)
\end{equation} induces a morphism of chain complexes from $(E_\bullet(A),\delta)$ to $(C^{hoc}_\bullet(A),b)$. This leads to a morphism on the respective homology groups, which
we continue to denote by $\varepsilon_\bullet:HE_\bullet(A)\longrightarrow HH_\bullet(A)$. 

\smallskip
Since $C$ is cocommutative, we  note that the measuring $\Phi:C\longrightarrow Hom_K(A,A')$ induces maps
\begin{equation}\label{lie2}
E_n^\Phi(x):E_n(A)\longrightarrow E_n(A') \qquad  a_0\otimes a_1\wedge ...\wedge a_n \mapsto x(a_0\otimes a_1\wedge ...\wedge a_n):=x_{(1)}(a_0)\otimes x_{(2)}(a_1)\wedge ...\wedge x_{(n+1)}(a_n)
\end{equation} for each $x\in C$. 

\begin{thm}\label{P2.4aj} Let $C$ be a cocommutative coalgebra and let $\Phi:C\longrightarrow Hom_K(A,A')$  be a measuring between algebras. For each $x\in C$ and $n\geq 1$, the following diagrams are commutative

\begin{equation}\label{213cd}
\begin{array}{lll}
\begin{CD}
E_n(A) @>\delta>> E_{n-1}(A)\\
@VE_n^\Phi(x) VV @VVE_{n-1}^\Phi(x) V \\
E_n(A') @>\delta>> E_{n-1}(A')\\
\end{CD}
& \qquad\qquad &
\begin{CD}
E_n(A) @>\varepsilon_n>> C_n(A)\\
@VE_n^\Phi(x) VV @VVC_n^\Phi(x) V \\
E_n(A') @>\varepsilon_n>> C_n(A')\\
\end{CD}
\end{array}
\end{equation} In particular, the measuring $\Phi:C\longrightarrow Hom_K(A,A')$ induces maps $HE^\Phi_n:HE_n(A)\longrightarrow HE_{n}(A')$, $n\geq 0$ which fit into  commutative diagrams
\begin{equation}
\begin{CD}
HE_n(A) @>\varepsilon_n>> HH_n(A)\\
@VHE_n^\Phi(x)VV @VVHH_n^\Phi(x)V \\
HE_n(A')@>\varepsilon_n>> HH_n(A')\\
\end{CD}
\end{equation} for each $x\in C$.
\end{thm}

\begin{proof} Since $C$ is cocommutative, we notice that for $x\in C$, and $a$, $b\in A$, we have
\begin{equation} x [a,b]=x(ab)-x(ba)=x_{(1)}(a)x_{(2)}(b)-x_{(1)}(b)x_{(2)}(a)=x_{(1)}(a)x_{(2)}(b)-x_{(2)}(b)x_{(1)}(a)=[x_{(1)}(a),x_{(2)}(b)]
\end{equation}
For  $x\in C$ and $a_0$, $a_1$,..., $a_n\in A$, we now see that
\begin{equation*}\small
\begin{array}{l}
x\cdot \delta(a_0\otimes a_1\wedge ...\wedge a_n)\\
=\underset{i=1}{\overset{n}{\sum}} (-1)^ix_{(1)}[a_0,a_i]\otimes x_{(2)}(a_1)\wedge ... \wedge \hat{a}_i \wedge ... \wedge x_{(n)}( a_n) + \underset{1\leq i\leq j\leq n}{{\sum}} (-1)^{i+j-1} x_{(1)}(a_0)\otimes x_{(2)}[a_i,a_j]\wedge x_{(3)}(a_1)\wedge ... \wedge \hat{a}_i\wedge ... \wedge \hat{a}_j\wedge ...\wedge x_{(n)}(a_n)\\
=\underset{i=1}{\overset{n}{\sum}} (-1)^i[x_{(1)}(a_0),x_{(2)}(a_i)]\otimes x_{(3)}(a_1)\wedge ... \wedge \hat{a}_i \wedge ... \wedge x_{(n+1)}( a_n) \\ \qquad + \underset{1\leq i\leq j\leq n}{{\sum}} (-1)^{i+j-1} x_{(1)}(a_0)\otimes [x_{(2)}(a_i),x_{(3)}(a_j)]\wedge x_{(4)}(a_1)\wedge ... \wedge \hat{a}_i\wedge ... \wedge \hat{a}_j\wedge ...\wedge x_{(n+1)}(a_n)\\
=\underset{i=1}{\overset{n}{\sum}} (-1)^i[x_{(1)}(a_0),x_{(i+1)}(a_i)]\otimes x_{(2)}(a_1)\wedge ... \wedge \hat{a}_i \wedge ... \wedge x_{(n+1)}( a_n) \\ \qquad + \underset{1\leq i\leq j\leq n}{{\sum}} (-1)^{i+j-1} x_{(1)}(a_0)\otimes [x_{(i+1)}(a_i),x_{(j+1)}(a_j)]\wedge x_{(2)}(a_1)\wedge ... \wedge \hat{a}_i\wedge ... \wedge \hat{a}_j\wedge ...\wedge x_{(n+1)}(a_n)\\
=\delta(x_{(1)}(a_0)\otimes x_{(2)}(a_1)\wedge ...\wedge x_{(n+1)}(a_n))
\\ \end{array}
\end{equation*} It follows that the left hand diagram in \eqref{213cd} commutes. Further, we also have
\begin{equation*}
\small
\begin{array}{l}
C_n^\Phi(x)(\varepsilon_n(a_0\otimes a_1\wedge ...\wedge a_n))=x(\varepsilon_n\cdot (a_0\otimes a_1\otimes ...\otimes a_n))=x\left(\underset{\sigma\in S_n}{\sum}sgn(\sigma)(a_0,a_{\sigma^{-1}(1)},...,a_{\sigma^{-1}(n)}) \right)\\=\underset{\sigma\in S_n}{\sum}sgn(\sigma)(x_{(1)}(a_0),x_{(2)}(a_{\sigma^{-1}(1)}),...,x_{(n+1)}(a_{\sigma^{-1}(n)}))
=\underset{\sigma\in S_n}{\sum}sgn(\sigma)(x_{(1)}(a_0),x_{\sigma^{-1}(1)+1}(a_{\sigma^{-1}(1)}),...,x_{\sigma^{-1}(n)+1}(a_{\sigma^{-1}(n)})) =\varepsilon_n(E_n^\Phi(a_0\otimes a_1\wedge ...\wedge a_n))) \\
\end{array}
\end{equation*} Hence, the right hand diagram in \eqref{213cd} also commutes. This proves the result. 
\end{proof}

\smallskip
In the rest of this section, we assume that the algebras $A$, $A'$ are commutative.   We will now study how measurings induce maps on the de Rham complex and relate them to the maps induced on Hochschild and cyclic complexes.  For a commutative $K$-algebra $A$, we set $\Omega_A^0:=A$. We  know that the $A$-module $\Omega_A^1$ of K$\ddot{\mbox{a}}$hler differentials is generated by the $K$-linear symbols $da$ with $a\in A$, subject to the relations $d(ab)=adb+bda$ for $a$, $b\in A$ (see \cite[$\S$ 1.1.9]{Lod}).  

\smallskip
Now let $C$ be a cocommutative coalgebra and let $\Phi:C\longrightarrow Hom_K(A,A')$ be a measuring between commutative algebras. For any $x\in C$, we set 
\begin{equation}\label{omeg2}  \Omega_\Phi^1(x):\Omega_A^1
\longrightarrow \Omega_{A'}^1 \qquad a_0da_1\mapsto x(a_0da_1):=x_{(1)}(a_0)d(x_{(2)}(a_1))
\end{equation} Since $\Phi:C\longrightarrow Hom_K(A,A')$ is a measuring and $C$ is cocommutative, we observe that  for $a_0$, $a_1\in A$, we have
\begin{equation}\label{omeg21}
\begin{array}{ll}
x(d(a_0a_1))=d(x(a_0a_1))=d(x_{(1)}(a_0)x_{(2)}(a_1))&=x_{(1)}(a_0)d(x_{(2)}(a_1))+x_{(2)}(a_1)d(x_{(1)}(a_0))\\
&=x_{(1)}(a_0)d(x_{(2)}(a_1))+x_{(1)}(a_1)d(x_{(2)}(a_0))\\
&=x(a_0da_1)+x(a_1da_0)\\
\end{array}
\end{equation} It follows from \eqref{omeg21} that the map $\Omega_\Phi^1(x)$  in \eqref{omeg2} is well defined. From \eqref{omeg2} and the fact that $\Phi$ is a measuring, it is also clear that $x(am)=x_{(1)}(a)x_{(2)}(m)$ for any $x\in C$, $a\in A$ and $m\in \Omega_A^1$. For each $x\in C$, we also set $\Omega^0_\Phi(x):=\Phi(x): \Omega_A^0=A
\longrightarrow A'=\Omega_{A'}^0$. 

\begin{lem}\label{L2.3} 
Let $C$ be a cocommutative coalgebra and let $\Phi:C\longrightarrow Hom_K(A,A')$ be a measuring between commutative algebras. For $x\in C$ and $p>1$, we have an induced map
\begin{equation}\label{2.11arr}
 \Omega_\Phi^p(x):\Omega_A^p=\wedge_A^p\Omega_A^1
\longrightarrow\wedge^p_{A'} \Omega_{A'}^1 =\Omega_{A'}^p \qquad a_0da_1...da_p\mapsto x(a_0da_1...da_p):=x_{(1)}(a_0)d(x_{(2)}(a_1))...d(x_{(p+1)}a_p)
\end{equation}
\end{lem}

\begin{proof}
For $x\in C$, we first claim that $\Omega^1_\Phi(x):\Omega_A^1
\longrightarrow \Omega_{A'}^1$ induces a map $\otimes^p_A\Omega_A^1\longrightarrow \otimes^p_{A'}\Omega_{A'}^1$ at the level of the $p$-fold tensor product, given by 
setting $x(m_1\otimes ... \otimes m_p):=x_{(1)}(m_1)\otimes ...\otimes x_{(p)}(m_p)$ for $m_1\otimes ...\otimes m_p\in \otimes_A^p\Omega_A^1$. Indeed, for any $a\in A$ and $1\leq j\leq p-1$, we have
\begin{equation}
\begin{array}{l}
x(m_1\otimes ... \otimes am_j\otimes m_{j+1}\otimes ...\otimes m_p) \\=x_{(1)}(m_1)\otimes ... \otimes x_{(j)}(am_j)\otimes x_{(j+1)}(m_{j+1})\otimes ...\otimes x_{(p)}(m_p) \\
=x_{(1)}(m_1)\otimes ... \otimes x_{(j)}(a)x_{(j+1)}(m_j)\otimes x_{(j+2)}(m_{j+1})\otimes ...\otimes x_{(p+1)}(m_p) \\
=x_{(1)}(m_1)\otimes ... \otimes x_{(j+1)}(m_j)\otimes x_{(j)}(a)x_{(j+2)}(m_{j+1})\otimes ...\otimes x_{(p+1)}(m_p) \\
=x_{(1)}(m_1)\otimes ... \otimes x_{(j)}(m_j)\otimes x_{(j+1)}(a)x_{(j+2)}(m_{j+1})\otimes ...\otimes x_{(p+1)}(m_p) \qquad\qquad \qquad\mbox{(since $C$ is cocommutative)}\\
=x_{(1)}(m_1)\otimes ... \otimes x_{(j)}(m_j)\otimes x_{(j+1)}(am_{j+1}) \otimes ...\otimes x_{(p)}(m_p)=x(m_1\otimes ... \otimes m_j\otimes am_{j+1}\otimes ...\otimes m_p) \\
\end{array}
\end{equation} Again since $C$ is cocommutative, for any permutation $\sigma\in S_p$, we note that
\begin{equation}\label{213po}
x(m_1\otimes ...\otimes m_p)-sgn(\sigma)x(m_{\sigma(1)}\otimes ...\otimes m_{\sigma(p)})=x_{(1)}(m_1)\otimes ...\otimes x_{(p)}(m_p)-sgn(\sigma)(x_{\sigma(1)}(m_{\sigma(1)})\otimes ...
\otimes x_{\sigma(p)}(m_{\sigma(p)}))
\end{equation} Since $K$ is a field of characteristic zero, it follows from \eqref{213po} that there is an induced map at the level of exterior products. The result of \eqref{2.11arr} is now
clear from the definition in \eqref{omeg2}.
\end{proof}

For the commutative $K$-algebra $A$, we now consider the complex $(\Omega^\bullet_A,0)$ with zero differential $\Omega^\bullet_A
\overset{0}{\longrightarrow} \Omega^{\bullet-1}_A$. We know (see \cite[$\S$ 1.3]{Lod}) that the maps
\begin{equation}\label{221maps}
\begin{array}{c}
\varepsilon_p: \Omega_A^p\longrightarrow A\otimes A^{\otimes p}=C_p(A)\qquad a_0da_1...da_p\mapsto \varepsilon_p\cdot (a_0\otimes a_1\otimes ...
\otimes a_p)\\
\pi_p: C_p(A)=A\otimes A^{\otimes p}\longrightarrow \Omega_A^p\qquad a_0\otimes a_1\otimes ...
\otimes a_p \mapsto a_0da_1...da_p\\
\end{array}
\end{equation}
induce morphisms of complexes between $(\Omega^\bullet_A,0)$ and $(C^{hoc}_\bullet(A),b)$.  We will now show that the respective maps induced on homology groups by $\varepsilon_\bullet$ and $\pi_\bullet$ are compatible with measurings. 

\begin{thm}\label{P2.5gx}
Let $C$ be a cocommutative coalgebra and let $\Phi:C\longrightarrow Hom_K(A,A')$ be a measuring between commutative algebras. For each $x\in C$, the following diagrams commute.
\begin{equation}\label{222qj}
\begin{array}{lll}
\begin{CD}
\Omega_A^\bullet @>\varepsilon_\bullet>> HH_\bullet(A) \\
@V\Omega^\bullet_\Phi(x)VV @VVHH_\bullet^\Phi(x) V \\
\Omega_{A'}^\bullet @>\varepsilon_\bullet>> HH_\bullet(A') \\
\end{CD}& \qquad\qquad\qquad &
\begin{CD}
HH_\bullet(A) @>\pi_\bullet>> \Omega_A^\bullet \\
@VHH_\bullet^\Phi(x) VV @VV \Omega^\bullet_\Phi(x) V \\
HH_\bullet(A')  @>\pi_\bullet>> \Omega_{A'}^\bullet  \\
\end{CD} \\
\end{array}
\end{equation} 
\end{thm}
\begin{proof}
We consider $x\in C$ and $a_0$, $a_1$,...,$a_p\in A$. Using the fact that $C$ is cocommutative, we verify that
\begin{equation*}
\begin{array}{ll}
x(\varepsilon_p(a_0da_1...da_p))=x( \varepsilon_p\cdot (a_0\otimes a_1\otimes ...
\otimes a_p))&=x\left(\underset{\sigma\in S_p}{\sum} sgn(\sigma) (a_0,a_{\sigma^{-1}(1)},...,a_{\sigma^{-1}(p)})\right)\\
&=\underset{\sigma\in S_p}{\sum}sgn(\sigma)(x_{(1)}(a_0),x_{(2)}(a_{\sigma^{-1}(1)}),...,x_{(p+1)}(a_{\sigma^{-1}(p)}))\\
&=\underset{\sigma\in S_p}{\sum}sgn(\sigma)(x_{(1)}(a_0),x_{\sigma^{-1}(1)+1}(a_{\sigma^{-1}(1)}),...,x_{\sigma^{-1}(p)+1}(a_{\sigma^{-1}(p)}))\\
&=\varepsilon_p(x_{(1)}(a_0)d(x_{(2)}(a_1))...d(x_{(p+1)}(a_p)))\\
\end{array}
\end{equation*} This shows that the left hand diagram in \eqref{222qj} commutes. Similarly, to show  that the right hand diagram in \eqref{222qj} commutes, we verify that
\begin{equation*}\small
x(\pi_p(a_0\otimes a_1\otimes  ...\otimes a_p))=x(a_0da_1...da_p)=x_{(1)}(a_0)d(x_{(2)}(a_1))...d(x_{(p+1)}(a_p))=\pi_p(x_{(1)}(a_0)\otimes x_{(2)}(a_1) \otimes ... \otimes x_{(p+1)}(a_p))
\end{equation*} This proves the result. 
\end{proof}

We now  let $HDR^\bullet_A$ denote the de Rham cohomology of $A$, obtained from the complex $(\Omega_A^\bullet,d)$ with differential given by 
\begin{equation} d:\Omega_A^p\longrightarrow \Omega_A^{p+1}\qquad a_0da_1...da_p\mapsto da_0da_1...da_p
\end{equation} Let 
$C$ be a cocommutative coalgebra and let $\Phi:C\longrightarrow Hom_K(A,A')$ be a measuring between commutative algebras.  From the definition in 
\eqref{2.11arr}, it is clear that for each $x\in C$, 
the maps $\Omega^\bullet_\Phi(x)$  induce morphisms $HDR_\Phi^\bullet(x):HDR_A^\bullet \longrightarrow HDR_{A'}^\bullet$ at the level of
de Rham cohomology.

\smallskip
To conclude this section, we come to the relation between cyclic homology and de Rham cohomology. It is well known (see \cite[$\S$ 2.1.7]{Lod}) that the cyclic homology groups of $A$ can also be computed as the total homology of the ``mixed complex'' $\mathscr BC_{\bullet\bullet}(A):=(C_\bullet(A),b,B)$ where $b$ is the Hochschild differential $b:=\underset{i=0}{\overset{n}{\sum}}(-1)^id_i:C_n(A)
\longrightarrow C_{n-1}(A)$ associated to the cyclic module $C_\bullet(A)$ and $B:C_n(A)\longrightarrow C_{n+1}(A)$ is the Connes' boundary operator given by  setting
\begin{equation}\label{connesB}
B:=(-1)^{n+1}(1-t_{n+1})t_{n+1}s_n(1+t_n+...+t_n^n):C_n(A)\longrightarrow C_{n+1}(A)
\end{equation} Further, we know (see \cite[$\S$ 2.3]{Lod}) that the maps $\pi_n/n!: C_n(A)\longrightarrow \Omega_A^n$ for $n\geq 0$ determine a morphism of bicomplexes $\bar{\pi}_{\bullet}: (C_\bullet(A),b,B)\longrightarrow (\Omega_A^\bullet,0,d)$ which induces a map
\begin{equation}
HC_n(A)\xrightarrow{\qquad\bar{\pi}_n\qquad}\Omega^n_A/d\Omega^{n-1}_A\oplus HDR_A^{n-2}\oplus HDR^{n-4}_A\oplus \dots
\end{equation} for each $n\geq 0$.

\begin{thm}
\label{P2.6t}  Let $C$ be a cocommutative coalgebra and let $\Phi:C\longrightarrow Hom_K(A,A')$ be a measuring between commutative algebras. For each $x\in C$, the induced maps fit into the following commutative diagram
\begin{equation}
\begin{CD}
HC_n(A) @>\bar{\pi}_n>> \Omega^n_A/d\Omega^{n-1}_A\oplus HDR_A^{n-2}\oplus HDR^{n-4}_A\oplus \dots\\
@VHC_n^\Phi (x) VV @VV\Omega^n_\Phi/d\Omega^{n-1}_\Phi(x)\oplus HDR_\Phi^{n-2}(x)\oplus HDR_\Phi^{n-4}(x)\oplus \dots V\\
HC_n(A') @>\bar{\pi}_n>> \Omega^n_{A'}/d\Omega^{n-1}_{A'}\oplus HDR_{A'}^{n-2}\oplus HDR^{n-4}_{A'}\oplus \dots\\
\end{CD}
\end{equation}
\end{thm}

\begin{proof}
We choose $x\in C$.  From \eqref{connesB}, it is clear that the Connes' boundary operator $B$ is obtained from the structure maps of the cyclic module associated to an algebra. From Proposition \ref{P1.1}, we know that the measuring $\Phi$ induces a morphism $C_\bullet^\Phi(x):C_\bullet(A)
\longrightarrow C_\bullet(A')$ of cyclic modules.  It follows that the maps $HC_n^\Phi(x): HC_n(A)\longrightarrow HC_n(A')$ on cyclic homology can also be obtained from the morphism of mixed complexes from $(C_\bullet(A),b,B)$ to $(C_\bullet(A'),b,B)$ induced by the maps   $C_\bullet^\Phi(x):C_\bullet(A)
\longrightarrow C_\bullet(A')$. From the proof of Proposition \ref{P2.5gx}, it is clear that we have $(\pi_n/n!)\circ C_n^\Phi(x)=\Omega^n_\Phi(x)(\pi_n/n!):
C_n(A)\longrightarrow \Omega_{A'}^n$. It follows that we have the following commutative diagram of mixed complexes.
\begin{equation}
\begin{CD}
 (C_\bullet(A),b,B)@>\pi_n/n!>> (\Omega_A^\bullet,0,d)\\
@VC_\bullet^\Phi (x) VV @VV\Omega^\bullet_\Phi (x) V\\
 (C_\bullet(A'),b,B)@>\pi_n/n!>>  (\Omega_{A'}^\bullet,0,d)\\
\end{CD}
\end{equation}
The result is now clear by taking homologies.
\end{proof}

\section{Measurings and the product in cyclic homology}

Let $A$ be a $K$-algebra. Since the ground field $K$ has characteristic zero, we set $\bar{A}:=A/K$. For $n\geq 0$, we set $\bar{C}_n(A):=A\otimes \bar{A}^{\otimes n}$. We continue to denote by $b:\bar{C}_\bullet(A)\longrightarrow 
\bar{C}_{\bullet -1}(A)$ the induced Hochschild differential and by $B:\bar{C}_\bullet(A)\longrightarrow \bar{C}_{\bullet+1}(A)$ the induced Connes operator. We know that the cyclic homology groups $HC_\bullet(A)$ of $A$
may also be computed from the normalized mixed complex ${\overline{\mathscr BC}}_{\bullet\bullet}(A):=(\bar{C}_\bullet(A),b,B)$:
\begin{equation}\label{mc3.1}
\begin{CD}
\bar{C}_2(A)=A\otimes \bar{A}^{\otimes 2} @<B<<  \bar{C}_1(A)=A\otimes \bar{A} @<B<< \bar{C}_0(A)=A\\
@VbVV @VbVV @. \\
\bar{C}_1(A)=A\otimes \bar{A} @<B<< \bar{C}_0(A)=A @. \\
@VbVV @. @. \\
 \bar{C}_0(A)=A  @. @.
\end{CD}
\end{equation}
Similarly, the Hochschild homology groups $HH_\bullet(A)$ can also be computed from the normalized Hochschild complex $(\bar{C}_\bullet(A),b)$, which appears as the columns in \eqref{mc3.1}. 

\begin{thm}
\label{P3.1} Let $C$ be a cocommutative coalgebra and let $\Phi:C\longrightarrow Hom_K(A,A')$ be a measuring between algebras. For each $x\in C$, there are well-defined maps
\begin{equation}
\bar{C}^\Phi_n(x):\bar{C}_n(A)\longrightarrow \bar{C}_n(A')\qquad (a_0,a_1,...,a_n)\mapsto x(a_0,a_1,...,a_n):=(x_{(1)}(a_0),x_{(2)}(a_1),...,x_{(n+1)}(a_n))
\end{equation} which induce a morphism $(\bar{C}_\bullet(A),b)\longrightarrow (\bar{C}_\bullet(A'),b)$ on normalized Hochschild complexes as well as on the normalized mixed complexes $(\bar{C}_\bullet(A),b,B)\longrightarrow (\bar{C}_\bullet(A'),b,B)$. The corresponding induced maps on Hochschild homology and on cyclic homology are identical to 
$HH_\bullet^\Phi(x)$ and $HC^\Phi_\bullet(x)$ respectively. 
\end{thm}

\begin{proof}
For any $i\geq 1$, we note that
\begin{equation}\label{eq3.2}
\begin{array}{ll}
x(a_0,...,a_{i-1},1,a_{i+1},...,a_n)&=(x_{(1)}(a_0),...,x_{(i)}(a_{i-1}),x_{(i+1)}(1),x_{(i+2)}(a_{i+1}),...,x_{(n+1)}(a_n))\\
&=\epsilon_C(x_{(i+1)})(x_{(1)}(a_0),...,x_{(i)}(a_{i-1}), 1,x_{(i+2)}(a_{i+1}),...,x_{(n+1)}(a_n))\\
\end{array}
\end{equation} From \eqref{eq3.2}, it is clear that the maps $C_n^\Phi(x):C_n(A)\longrightarrow C_n(A')$ descend to maps $\bar{C}^\Phi_n(x):\bar{C}_n(A)\longrightarrow \bar{C}_n(A')$. From Proposition \ref{P1.1}, we know that the maps $C_\bullet^\Phi(x)$ are well behaved with respect to the operators $b$ and $B$.  Now since the canonical maps $C_\bullet(A)\longrightarrow \bar{C}_\bullet(A)$ and $C_\bullet(A')\longrightarrow \bar{C}_\bullet(A')$ are epimorphisms, it follows that the following diagrams commute:
\begin{equation}
\begin{array}{ccc}
\begin{CD}
\bar{C}_\bullet(A) @>b>> \bar{C}_{\bullet-1}(A)\\
@V\bar{C}_\bullet^\Phi(x)VV @VV\bar{C}_{\bullet-1}^\Phi(x)V \\
\bar{C}_\bullet(A') @>b>> \bar{C}_{\bullet-1}(A')\\
\end{CD}
& \qquad\qquad & 
\begin{CD}
\bar{C}_\bullet(A) @>B>> \bar{C}_{\bullet+1}(A)\\
@V\bar{C}_\bullet^\Phi(x)VV @VV\bar{C}_{\bullet+1}^\Phi(x)V \\
\bar{C}_\bullet(A') @>B>> \bar{C}_{\bullet+1}(A')\\
\end{CD}\\
\end{array}
\end{equation}
Accordingly, we have  maps  $(\bar{C}_\bullet(A),b)\longrightarrow (\bar{C}_\bullet(A'),b)$ and $(\bar{C}_\bullet(A),b,B)\longrightarrow (\bar{C}_\bullet(A'),b,B)$ of complexes induced by the measuring $\Phi$. Finally, since the canonical maps $C_\bullet(A)\longrightarrow \bar{C}_\bullet(A)$ and $C_\bullet(A')\longrightarrow \bar{C}_\bullet(A')$ induce isomorphisms at the homology level, it follows that the induced maps on Hochschild homology and on cyclic homology are identical to 
$HH_\bullet^\Phi(x)$ and $HC^\Phi_\bullet(x)$ respectively. 
\end{proof}

\begin{rem}\label{R3.2} \emph{For $n\geq 0$, let $\tilde{C}_n(A):=A^{\otimes n+1}/(1-t)$, where $t$ is the cyclic operator as described in Section 2. Then, the Hochschild differential $b$ descends to the quotients $\tilde{C}_\bullet(A)$ and the cyclic homology of $A$ may be computed by means of the complex $(\tilde C_\bullet(A),b)$ (see \cite[$\S$ 2.1.4]{Lod}). By similar reasoning as in the proof of Proposition \ref{P3.1}, we can verify that the measuring $\Phi:C\longrightarrow Hom_K(A,A')$ induces a morphism of complexes  $\tilde C_\bullet^\Phi(x):\tilde C_\bullet(A)\longrightarrow 
\tilde C_\bullet(A')$ for each $x\in C$. Further, the induced maps on homology groups are identical to $HC_\bullet^\Phi(x)$, for each $x\in C$.  }

\end{rem}

Suppose now that $A$ is commutative. For $p$, $q\geq 1$, a permutation $\sigma\in S_{p+q}$ is said to be a $(p,q)$ shuffle if it satisfies $\{\sigma(1)<\dots <\sigma(p)\}$ and $\{\sigma(p+1)<\dots <\sigma(p+q)\}$. 
We let $S_{p,q}$ denote the set of all $(p,q)$-shuffles. Then, we know (see \cite[$\S$ 4.2]{Lod}) that the operation
\begin{equation}\label{shf3.3}
\begin{array}{c}
\_\_ \times \_\_ : C_p(A)\otimes C_q(A)\longrightarrow C_{p+q}(A)\\ \\
(a_0,a_1,...,a_p)\otimes (b_0,b_1,...,b_q)\mapsto \underset{\sigma\in S_{p,q}}{\mathlarger \sum} sgn(\sigma) \sigma\cdot (a_0b_0,a_1,...,a_p,b_1,...,b_q)
\\
\end{array}
\end{equation} induces a product $\_\_ \times \_\_: HH_p(A)\otimes HH_q(A)\longrightarrow HH_{p+q}(A)$ that makes the Hochschild homology of  the commutative $K$-algebra $A$ into a graded ring $(HH_\bullet(A),\times)$.

\smallskip
In the rest of this section, we will always assume that the algebras $A$, $A'$ are commutative. In \cite[Proposition 2.3]{BK1}, it was shown that if $\Phi:C\longrightarrow Hom_K(A,A')$ is a measuring between commutative algebras, then we have an induced measuring
\begin{equation}\label{3.6ms}
\Phi^\times: C\longrightarrow Hom_K(HH_\bullet(A),HH_\bullet(A')) \qquad x\mapsto HH_\bullet^\Phi(x):=\{HH_n^\Phi(x):HH_n(A)\longrightarrow HH_n(A')\}_{n\geq 0}
\end{equation}
between the algebras  $(HH_\bullet(A),\times)$ and $(HH_\bullet(A'),\times)$. In other words, for each $x\in C$, we can write
\begin{equation}\label{klausmet}
x\cdot ((a_0,a_1,...,a_p)\times (b_0,b_1,...,b_q))=x_{(1)}(a_0,a_1,...,a_p)\times x_{(2)}(b_0,b_1,...,b_q)
\end{equation} for $(a_0,a_1,...,a_p)\in C_p(A)$ and $(b_0,b_1,...,b_q)\in C_q(A)$. We will now do something similar for the product structure on cyclic homology. For this, we recall (see 
\cite[Theorem 3.3]{LQ} or \cite[$\S$ 4.4]{Lod}) that the operation on the total complexes
\begin{equation}\label{ast3}
\begin{array}{c}
\_\_\ast\_\_: Tot(\overline{\mathscr BC}_{\bullet\bullet} (A))_p \otimes Tot(\overline{\mathscr BC}_{\bullet\bullet}(A))_q \longrightarrow Tot(\overline{\mathscr BC}_{\bullet\bullet}(A))_{p+q+1} \qquad 
\tilde{a}\otimes \tilde{b}\mapsto \tilde{a}\ast\tilde{b} \\  \\
(\tilde{a}_p,\tilde{a}_{p-2},...)\otimes (\tilde{b}_q,\tilde{b}_{q-2},...)\mapsto (B\tilde{a}_p\times \tilde{b}_q,B\tilde{a}_{p}\times \tilde{b}_{q-2},...)  \\
\end{array}
\end{equation} for $\tilde{a}=(\tilde{a}_p,\tilde{a}_{p-2},...)\in Tot(\overline{\mathscr BC}_{\bullet\bullet}(A))_p=\bar{C}_p(A)\oplus \bar{C}_{p-2}(A)\oplus ...$ and 
$\tilde{b}=(\tilde{b}_q,\tilde{b}_{q-2},...)\in Tot(\overline{\mathscr BC}_{\bullet\bullet}(A))_q=\bar{C}_q(A)\oplus \bar{C}_{q-2}(A)\oplus ...$ induces a product structure
\begin{equation}\label{hcprod}
\_\_\ast\_\_:  HC_{p}(A)\otimes HC_q(A)\longrightarrow HC_{p+q+1}(A)
\end{equation} By a shift in degree, setting $HC'_\bullet(A):=HC_{\bullet-1}(A)$, we know (see \cite[$\S$ 3]{LQ}) that the product in \eqref{hcprod} makes cyclic homology into an associative graded $K$-algebra $(HC'_\bullet(A),\ast)$.

\begin{Thm}\label{T3.2}
Let $C$ be a cocommutative coalgebra and let $\Phi:C\longrightarrow Hom_K(A,A')$ be a measuring between commutative algebras. Then, the maps
\begin{equation}
HC'^\Phi_\bullet(x):HC'_\bullet(A)=HC_{\bullet-1}(A)\xrightarrow{\qquad HC_{\bullet-1}^\Phi(x)\qquad}HC_{\bullet-1}(A')=HC'_{\bullet}(A') \qquad x\in C
\end{equation} determine a measuring
\begin{equation}
\Phi^\ast: C\longrightarrow Hom_K(HC'_\bullet(A),HC'_\bullet(A'))\qquad x\mapsto \Phi^\ast(x):=HC'^\Phi_\bullet(x):HC'_\bullet(A)\longrightarrow HC'_\bullet(A')
\end{equation} of algebras from $(HC'_\bullet(A),\ast)$ to $(HC'_\bullet(A'),\ast)$.
\end{Thm}

\begin{proof}
Let $x\in C$. From Proposition \ref{P3.1}, we know that the operators $\bar{C}^\Phi_n(x):\bar{C}_n(A)\longrightarrow \bar{C}_n(A')$ are well behaved with respect to the Connes' boundary operators
$B$ on the normalized mixed complexes $\overline{\mathscr BC}_{\bullet\bullet}(A)$ and $\overline{\mathscr BC}_{\bullet\bullet}(A')$. We now consider $\tilde{a}=(\tilde{a}_p,\tilde{a}_{p-2},...)\in Tot(\overline{\mathscr BC}_{\bullet\bullet}(A))_p=\bar{C}_p(A)\oplus \bar{C}_{p-2}(A)\oplus ...$ and 
$\tilde{b}=(\tilde{b}_q,\tilde{b}_{q-2},...)\in Tot(\overline{\mathscr BC}_{\bullet\bullet}(A))_q=\bar{C}_q(A)\oplus \bar{C}_{q-2}(A)\oplus ...$ as in \eqref{ast3}. Using \eqref{klausmet} and the definition of the product $\ast$ in \eqref{ast3}, we see that
\begin{equation*}
\begin{array}{ll}
x(\tilde a\ast \tilde b)=x((\tilde{a}_p,\tilde{a}_{p-2},...)\ast (\tilde{b}_q,\tilde{b}_{q-2},...)) 
&=  x(B\tilde{a}_p\times \tilde{b}_q,B\tilde{a}_{p}\times \tilde{b}_{q-2},...) \\
&= (x(B\tilde{a}_p\times \tilde{b}_q),x(B\tilde{a}_{p}\times \tilde{b}_{q-2}),...) \\
&=(x_{(1)}(B\tilde a_p)\times x_{(2)}(\tilde b_q),x_{(1)}(B\tilde a_p)\times x_{(2)}(\tilde b_{q-2}),...)\\
&=(B(x_{(1)}(\tilde a_p))\times x_{(2)}(\tilde b_q),B(x_{(1)}(\tilde a_p))\times x_{(2)}(\tilde b_{q-2}),...)\\
&=(x_{(1)}(\tilde a_p),x_{(1)}(\tilde a_{p-2}),...)\ast (x_{(2)}(\tilde b_q),x_{(2)}(\tilde b_{q-2}),...) \\
&=x_{(1)}(\tilde a)\ast x_{(2)}(\tilde b)\\
\end{array}
\end{equation*} The result is now clear. 
\end{proof}

\begin{rem}
\emph{We note that the ring $(HC'_\bullet(A),\ast)$ is non-unital, since $HC'_0(A)=0$. Accordingly, the measuring in Theorem \ref{T3.2} is also a measuring of non-unital algebras, i.e., it satisfies only the condition in \eqref{ms1.1} involving the coproduct on $C$.}
\end{rem} 

We know (see \cite[$\S$ 4.4.4]{Lod}) that the  maps $B:HC'_p(A)=HC_{p-1}(A)\longrightarrow HH_p(A)$ appearing in the periodicity sequence \eqref{pr2.6} 
determine a map of graded algebras from $(HC'_\bullet(A),\ast)$ to $(HH_\bullet(A),\times)$. At the chain  level, this map is induced by the operator $B$ as follows
\begin{equation}\label{alconB}
\begin{array}{c}
Tot(\overline{\mathscr BC}_{\bullet\bullet} (A))_{p-1} =\bar{C}_{p-1}(A)\oplus \bar{C}_{p-3}(A)\oplus ... \longrightarrow \bar{C}_{p}(A)\qquad \qquad
(\tilde a_{p-1},\tilde a_{p-3},...)\mapsto B\tilde a_{p-1}
\end{array}
\end{equation} Combining with the shuffle product on the Hochschild homology ring $(HH_\bullet(A),\times)$, it follows that $HH_\bullet(A)$ becomes a module over $(HC'_\bullet(A),\ast)$  with action determined by
\begin{equation}\label{alconB1}
\begin{array}{c}
\ast: HC'_p(A) \otimes HH_q(A)\longrightarrow HH_{p+q}(A)\qquad\qquad
 (\tilde a_{p-1},\tilde a_{p-3},...)\otimes \tilde b_q \mapsto B\tilde a_{p-1}\times \tilde b_q
\end{array}
\end{equation}
If $D$ is a $C$-comodule, and $M$, $M'$ are modules over $A$ and $A'$ respectively, we recall that a comodule measuring (see for instance, \cite{Bat}, \cite{Vas1}) with respect to 
$(C,\Phi)$ consists of a map
\begin{equation}\label{mescom}
\begin{array}{c}
\Psi: D\longrightarrow Hom_K(M,M')\qquad \Psi(d)(am)=d(am)= d_{(1)}(a)d_{(2)}(m)=\Phi(d_{(1)})(a)\Psi(d_{(2)})(m)\quad d\in D, a\in A, m\in M\\
\end{array}
\end{equation}
where the $C$-comodule structure on $D$ is given by $d\mapsto  d_{(1)}\otimes d_{(2)}\in C\otimes D$ in Sweedler notation, with summation signs suppressed.

\begin{thm}\label{P3.3st}
Let $C$ be a cocommutative coalgebra and let $\Phi:C\longrightarrow Hom_K(A,A')$ be a measuring between commutative algebras. Then, the maps
\begin{equation}\label{mescom3.15}
\begin{array}{c}
\Phi^\ast:  C\longrightarrow Hom_K(HC'_\bullet(A),HC'_\bullet(A'))\qquad x\mapsto \Phi^\ast(x):=HC'^\Phi_\bullet(x):HC'_\bullet(A)\longrightarrow HC'_\bullet(A') \\
\Phi^\times:  C\longrightarrow Hom_K(HH_\bullet(A),HH_\bullet(A')) \qquad x\mapsto \Phi^\times (x)=HH_\bullet^\Phi(x):HH_\bullet(A)\longrightarrow HH_\bullet(A')\\
\end{array}
\end{equation} determine a comodule measuring with respect to $(C,\Phi^\ast)$ from the $(HC'_\bullet(A),\ast)$-module $HH_\bullet(A)$ to the $(HC'_\bullet(A'),\ast)$-module $HH_\bullet(A')$, where $C$ is treated as a $C$-comodule.
\end{thm}

\begin{proof}
We consider $x\in C$ along with  $\tilde{a}=(\tilde{a}_{p-1},\tilde{a}_{p-3},...)\in Tot(\overline{\mathscr BC}_{\bullet\bullet}(A))_{p-1}=\bar{C}_{p-1}(A)\oplus \bar{C}_{p-3}(A)\oplus ...$ and $\tilde b_q\in \bar{C}_q(A)$. Again, we know from  Proposition \ref{P3.1} that the operators $\bar{C}^\Phi_n(x):\bar{C}_n(A)\longrightarrow \bar{C}_n(A')$ are well behaved with respect to the  Connes operators
$B$ on the normalized mixed complexes. Combining with the fact that $\Phi^\times:  C\longrightarrow Hom_K(HH_\bullet(A),HH_\bullet(A'))$ is a measuring of algebras with respect to the shuffle product $\times$, we have
\begin{equation*}
\begin{array}{ll}
\Phi^\times(x)((\tilde a_{p-1},\tilde a_{p-3},...)\ast\tilde b_q )=x((\tilde a_{p-1},\tilde a_{p-3},...)\ast \tilde b_q )& =x( B\tilde a_{p-1}\times \tilde b_q)\\
&= x_{(1)}(B\tilde a_{p-1})\times x_{(2)}(\tilde b_q)\\
&=B(x_{(1)}(\tilde a_{p-1}))\times x_{(2)}(\tilde b_q) \\
&=x_{(1)}(\tilde a_{p-1},\tilde a_{p-3},...)\ast x_{(2)}(\tilde b_q) \\
&= \Phi^\ast(x_{(1)})(\tilde a_{p-1},\tilde a_{p-3},...)\ast \Phi^\times(x_{(2)})(\tilde b_q)\\
\end{array}
\end{equation*} This proves the result.
\end{proof}

Given $K$-algebras $R$, $R'$ (not necessarily commutative), we know (see Sweedler \cite[Chapter VII]{Sweed}) that there is a  coalgebra $\mathcal M(R,R')$ and a measuring $\Psi(R,R'):
\mathcal M(R,R')\longrightarrow Hom_K(R,R')$ that is universal in the category of coalgebra measurings from $R$ to $R'$. In other words, given any coalgebra measuring 
$\Phi:C\longrightarrow Hom_K(R,R')$, there exists a unique coalgebra map $\zeta: C\longrightarrow \mathcal M(R,R')$ such that $\Phi=\Psi(R,R')\circ\zeta$. The coalgebra 
$\mathcal M(R,R')$ is referred to as the universal measuring coalgebra from $R$ to $R'$. It is well known (see for instance, \cite{AJ}) that this leads to an enrichment of $K$-algebras over $K$-coalgebras. In other words, one can define a category $ALG_K$ whose objects are $K$-algebras and whose Hom objects are given by 
$
ALG_K(R,R'):=\mathcal M(R,R')
$ for $K$-algebras $R$, $R'$. 

\smallskip
In the same manner, we have a measuring $\Psi^c(R,R'):
\mathcal M_c(R,R')\longrightarrow Hom_K(R,R')$ that is universal in the category of cocommutative coalgebra measurings from $R$ to $R'$, i.e., coalgebra measurings $\Phi:C\longrightarrow Hom_K(R,R')$ such that $C$ is cocommutative. We note that  $\mathcal M_c(R,R')$ is the cocommutative part of the universal measuring coalgebra $\mathcal M(R,R')$ (see 
\cite[Proposition 1.4]{GM1}). Accordingly, we can also define the category $ALG_K^c$ whose objects are $K$-algebras and whose Hom objects are given by 
$
ALG_K^c(R,R'):=\mathcal M_c(R,R')
$ for $K$-algebras $R$, $R'$. 

\smallskip
 We also consider the full subcategory $cALG_K$ of $ALG^c_K$ whose objects are commutative $K$-algebras and whose Hom objects are given by 
\begin{equation}
cALG_K(A,A'):=\mathcal M_c(A,A')
\end{equation} for commutative $K$-algebras $A$, $A'$. This gives us an enrichment of commutative $K$-algebras over the symmetric monoidal category $coCoalg_K$ of cocommutative $K$-coalgebras.

\smallskip
It is easy to see that the above constructions extend to algebras that are not necessarily unital. Given a commutative $K$-algebra $A$, we   consider the ring $(HC'_\bullet(A),\ast)$ as described above. We now define a new category $c\widetilde{ALG}_K$ whose objects are commutative $K$-algebras and whose Hom objects are given by
\begin{equation}
c\widetilde{ALG}_K(A,A'):=\mathcal M_c((HC'_\bullet(A),\ast),(HC'_\bullet(A'),\ast))
\end{equation} for commutative $K$-algebras $A$, $A'$. This gives us another enrichment of commutative $K$-algebras over the symmetric monoidal category $coCoalg_K$ of cocommutative $K$-coalgebras.

\begin{thm}
\label{P3.5} Let $A$, $A'$ be commutative $K$-algebras. Then, there are canonical maps 
\begin{equation}
\tau(A,A'): \mathcal M_c(A,A')\longrightarrow \mathcal M_c((HC'_\bullet(A),\ast),(HC'_\bullet(A'),\ast))
\end{equation}
of $K$-coalgebras.
\end{thm}

\begin{proof}
For the commutative algebras $A$, $A'$, we  consider the   measuring 
\begin{equation}\label{320g}
\Psi^c(A,A'):
\mathcal M_c(A,A')\longrightarrow Hom_K(A,A')
\end{equation} that is universal among cocommutative measurings from $A$ to $A'$. Since $\mathcal M_c(A,A')$ is cocommutative, it follows from Theorem \ref{T3.2} that there is an induced measuring
\begin{equation}
\Psi^{c\ast}(A,A'): \mathcal M_c(A,A')\longrightarrow Hom_K(HC'_\bullet(A),HC'_\bullet(A'))
\end{equation} between corresponding cyclic homology algebras. The coalgebra map $\tau(A,A'): \mathcal M_c(A,A')\longrightarrow \mathcal M_c((HC'_\bullet(A),\ast),(HC'_\bullet(A'),\ast))$ now follows directly from the universal property of $\mathcal M_c((HC'_\bullet(A),\ast),(HC'_\bullet(A'),\ast))$. 
\end{proof}

\begin{Thm}
\label{T3.6} There is a $coCoalg_K$-enriched functor $cALG_K\longrightarrow c\widetilde{ALG}_K$  which is identity on objects and whose mapping on Hom objects is given by \begin{equation}\label{3.21tu}
\tau(A,A'): \mathcal M_c(A,A')\longrightarrow \mathcal M_c((HC'_\bullet(A),\ast),(HC'_\bullet(A'),\ast))
\end{equation} for commutative algebras $A$, $A'$. 
\end{Thm}

\begin{proof}
We will first show that the morphisms in \eqref{3.21tu} are compatible with composition in the categories $cALG_K$ and $c\widetilde{ALG}_K$. In other words, for commutative algebras $A$, $A'$ and $A''$, we have to show that the following diagram is commutative
\begin{equation}\label{cd3.22}
\begin{CD}
\mathcal M_c(A,A') \otimes \mathcal M_c(A',A'') @>\circ >> \mathcal M_c(A,A'')\\
@V\tau(A,A')\otimes \tau (A',A'')VV @VV\tau(A,A'')V \\
\mathcal M_c(HC'_\bullet(A),HC'_\bullet(A')) \otimes \mathcal M_c(HC'_\bullet(A'),HC'_\bullet(A'')) @>\circ >> \mathcal M_c(HC'_\bullet(A),HC'_\bullet(A''))\\
\end{CD}
\end{equation} By definition, the horizontal morphisms in \eqref{cd3.22} are given by composition in the categories $ cALG_K$ and $c\widetilde{ALG}_K$, which are enriched over $coCoalg_K$. Hence, these are maps of cocommutative coalgebras. From Proposition \ref{P3.5}, it follows that the maps $\tau(A,A')$, $\tau(A',A'')$ and $\tau(A,A'')$ appearing in \eqref{cd3.22} are also morphisms in $coCoalg_K$. 

\smallskip In order to show that \eqref{cd3.22} commutes, we see from the  universal property of $\mathcal M_c(HC'_\bullet(A),HC'_\bullet(A''))$ that it suffices to show that the following two compositions are equal
\begin{equation}\label{equal322}
\begin{array}{lll}
\begin{CD}
\mathcal M_c(A,A') \otimes \mathcal M_c(A',A'') \\
@V\circ VV \\
 \mathcal M_c(A,A'') \\
@V\tau(A,A'')VV\\
 \mathcal M_c(HC'_\bullet(A),HC'_\bullet(A''))\\ 
@V\Psi^c(HC'_\bullet(A),HC'_\bullet(A''))VV \\
Hom_K(HC'_\bullet(A),HC'_\bullet(A''))\\
\end{CD} & \qquad\qquad &
\begin{CD}
\mathcal M_c(A,A') \otimes \mathcal M_c(A',A'')  \\
@VV\tau(A,A')\otimes \tau (A',A'')V \\
\mathcal M_c(HC'_\bullet(A),HC'_\bullet(A')) \otimes \mathcal M_c(HC'_\bullet(A'),HC'_\bullet(A'')) \\
@VV\circ V\\
\mathcal M_c(HC'_\bullet(A),HC'_\bullet(A'')) \\
@VV\Psi^c(HC'_\bullet(A),HC'_\bullet(A''))V \\
Hom_K(HC'_\bullet(A),HC'_\bullet(A''))\\
\end{CD} \\
\end{array}
\end{equation}
Here $ \Psi^c(HC'_\bullet(A),HC'_\bullet(A'')): \mathcal M_c(HC'_\bullet(A),HC'_\bullet(A''))\longrightarrow Hom_K(HC'_\bullet(A),HC'_\bullet(A''))$  is the universal cocommutative measuring from $(HC'_\bullet(A),\ast)$ to $(HC'_\bullet(A''),\ast)$. For the sake of convenience, let us denote by $\phi_1$ the left vertical composition and by $\phi_2$ the right vertical composition in \eqref{equal322}.

\smallskip
We now consider  $x\in \mathcal M_c(A,A')$ and $y\in  \mathcal M_c(A',A'')$, and their composition $y\circ x\in \mathcal M_c(A,A'')$.  Since $\circ:\mathcal M_c(A,A') \otimes \mathcal M_c(A',A'') \longrightarrow \mathcal M_c(A,A'')$ is a morphism of coalgebras, we note that
\begin{equation}\label{cop3.24}
(y\circ x)_{(1)}\otimes ... \otimes (y\circ x)_{(p)}=\Delta^p(y\circ x)=(y_{(1)}\circ x_{(1)})\otimes ... \otimes (y_{(p)}\circ x_{(p)})
\end{equation} for any $p\geq 0$. Accordingly, for any $(a_0,...,a_{p-1})\in \bar{C}_{p-1}(A)$, we have
\begin{equation}\label{325ar}
\begin{array}{ll}
\phi_1(x\otimes y)(a_0,...,a_{p-1})=((y\circ x)_{(1)}(a_0),..., (y\circ x)_{(p)}(a_{p-1)})&=((y_{(1)}\circ x_{(1)})(a_0),... ,(y_{(p)}\circ x_{(p)})(a_{p-1}))\\
&= (y_{(1)}(x_{(1)}(a_0)),...,y_{(p)}(x_{(p)}(a_{p-1})))\\
\end{array}
\end{equation} To describe the right vertical composition $\phi_2$ in \eqref{equal322}, we note that the following diagram commutes
\begin{equation*}\small
\xymatrix{
\mathcal M_c(A,A') \otimes \mathcal M_c(A',A'')  \ar[d]^{\Psi^c(HC'_\bullet(A),HC'_\bullet(A'))\otimes \Psi^c(HC'_\bullet(A'),HC'_\bullet(A''))}\ar[r]^{\tau(A,A')\otimes \tau (A',A'')\quad\quad\quad\quad\quad} &\mathcal M_c(HC'_\bullet(A),HC'_\bullet(A')) \otimes \mathcal M_c(HC'_\bullet(A'),HC'_\bullet(A''))  \ar[r]^{\qquad\qquad\quad\circ} &\mathcal M_c(HC'_\bullet(A),HC'_\bullet(A'')) \ar[d]_{\Psi^c(HC'_\bullet(A),HC'_\bullet(A''))}\\
Hom_K(HC'_\bullet(A),HC'_\bullet(A')) \otimes Hom_K(HC'_\bullet(A'),HC'_\bullet(A'')) \ar[rr]^{\qquad\circ} && Hom_K(HC'_\bullet(A),HC'_\bullet(A'')) \\
}
\end{equation*} It now follows that
\begin{equation}\label{327ar}
\phi_2(x\otimes y)(a_0,...,a_{p-1})= y(x(a_0,...,a_{p-1}))= (y_{(1)}(x_{(1)}(a_0)),...,y_{(p)}(x_{(p)}(a_{p-1})))
\end{equation} for $(a_0,...,a_{p-1})\in \bar{C}_{p-1}(A)$. From \eqref{325ar} and \eqref{327ar}, we have
$ \phi_1(x\otimes y)(a_0,...,a_{p-1})=\phi_2(x\otimes y)(a_0,...,a_{p-1})$ for $(a_0,...,a_{p-1})\in \bar{C}_{p-1}(A)$. Since $Tot(\overline{\mathscr BC}_{\bullet\bullet}(A))_{p-1} $ decomposes as  $Tot(\overline{\mathscr BC}_{\bullet\bullet}(A))_{p-1} =\bar{C}_{p-1}(A)\oplus \bar{C}_{p-3}(A)\oplus ...$, repeating the same reasoning for each direct summand shows that
$ \phi_1(x\otimes y)(\tilde a)=\phi_2(x\otimes y)(\tilde a)$ for any $\tilde a=(\tilde a_{p-1},\tilde a_{p-3},...)\in Tot(\overline{\mathscr BC}_{\bullet\bullet}(A))_{p-1}$. Hence, the diagram \eqref{cd3.22} commutes. 

\smallskip
To show  that we have a $coCoalg_K$-enriched functor $cALG_K\longrightarrow c\widetilde{ALG}_K$, it remains to check that the mapping on Hom objects in \eqref{3.21tu} is compatible with unit maps. Since $\Delta^p(1)=1\otimes ... \otimes 1$ ($(p+1)$-times) in the coalgebra $K$, it is clear from the definition of the action in \eqref{2.4mao} that the composition
\begin{equation}
K \longrightarrow \mathcal M_c(A,A) \xrightarrow{\qquad\tau(A,A)\qquad} \mathcal M_c(HC'(A),HC'(A))\xrightarrow{\qquad\Psi^c(HC'_\bullet(A),HC'_\bullet(A))\qquad}Hom_K(HC'_\bullet(A),HC'_\bullet(A))
\end{equation} is identical to the obvious map $K\longrightarrow Hom_K(HC'_\bullet(A),HC'_\bullet(A))$ given by scalar multiplication. Since the scalar multiplication is also a measuring,  the latter map must factor through the universal cocommutative measuring coalgebra $ \mathcal M_c(HC'(A),HC'(A))$. From the universal property of the measuring $\Psi^c(HC'_\bullet(A),HC'_\bullet(A)):  \mathcal M_c(HC'(A),HC'(A))\longrightarrow Hom_K(HC'_\bullet(A),HC'_\bullet(A))$, the result follows.
\end{proof} 

\section{Measurings and $\lambda$-decomposition on Hochschild and cyclic homology}

For a commutative algebra $A$, we recall that we have the $\lambda$-decomposition 
 of its Hochschild homology groups (see, for instance, \cite[$\S$ 4.5]{Lod})
\begin{equation}\label{4.1spl}
HH_n(A)=HH_n^{(1)}(A)\oplus ... \oplus HH_n^{(n)}(A)
\end{equation} for $n\geq 1$. Further, we know that the direct summands in \eqref{4.1spl} are compatible with the shuffle product  described in \eqref{shf3.3}.  In other words, the shuffle product restricts to give
\begin{equation}\label{4.2shf}
\_\_\times \_\_: HH_p^{(i)}(A)\times HH_q^{(j)}(A)\longrightarrow HH_{p+q}^{(i+j)}(A)
\end{equation}
As in Section 3, we let $C$ be a cocommutative coalgebra and let $\Phi: C\longrightarrow Hom_K(A,A')$ be  a measuring between commutative algebras $A$, $A'$. 
In \cite{BK1}, we have already shown that the maps $\{C_n^\Phi(x): C_n(A)\longrightarrow C_n(A')\}_{n\geq 0}$   induce morphisms $
HH^\Phi_\bullet(x): HH_\bullet(A)\longrightarrow HH_\bullet(A')$ on Hochschild homology.  In this section, we will show that these morphisms are well behaved with respect to the decomposition in \eqref{4.1spl} and the shuffle product in \eqref{4.2shf}. For this, we will first need some intermediate results on commutative  Hopf algebras.

\smallskip
Let $ \mathcal H$ be a commutative Hopf algebra over $K$, equipped with coproduct $\Delta_{\mathcal H}:\mathcal H\longrightarrow \mathcal H\otimes 
\mathcal H$ and counit $\epsilon_{\mathcal H}:\mathcal H\longrightarrow K$. If $f$, $g$ are $K$-linear endomorphisms of $\mathcal H$, their convolution $f\odot g$ is given by
\begin{equation}\label{4.3conv}
f\odot g:=\mu_{\mathcal H}(f\otimes g)\Delta_{\mathcal H}: \mathcal H\xrightarrow{\Delta_{\mathcal H}}\mathcal H\otimes 
\mathcal H\xrightarrow{f\otimes g}\mathcal H\otimes \mathcal H\xrightarrow{\mu_{\mathcal H}}\mathcal H
\end{equation} where $\mu_{\mathcal H}$ denotes the multiplication on $\mathcal H$. The convolution in \eqref{4.3conv} makes the collection 
$End_K(\mathcal H)$ of $K$-linear endomorphisms of $\mathcal H$ into an associative $K$-algebra (see \cite[$\S$ 4.5.2]{Lod}). The unit for the convolution product $\odot$ of endomorphisms in 
\eqref{4.3conv} is given by the composition $\iota_{\mathcal H}\epsilon_{\mathcal H}:\mathcal H\xrightarrow{\epsilon_{\mathcal H}}K
\xrightarrow{\iota_{\mathcal H}}\mathcal H$, where $\iota_{\mathcal H}$ is the unit map of the Hopf algebra $\mathcal H$.

\begin{lem}\label{Lm4.1}
Let $\mathcal H$, $\mathcal H'$ be commutative Hopf algebras. Let $C$ be a cocommutative coalgebra and let $\Psi:C\longrightarrow Hom_K(\mathcal H,\mathcal H')$ be a measuring from $\mathcal H$ to $\mathcal H'$ at the level of algebras.  Suppose we have $f$, $g\in End_K(\mathcal H)$ and $f'$, $g'\in End_K(\mathcal H')$ such that the following diagrams commute  for each $x\in C$. 
\begin{equation}\label{44cda}
\begin{array}{ccc}
\begin{CD}
\mathcal H @>f>> \mathcal H\\
@V\Psi(x)VV @VV\Psi(x)V\\
\mathcal H' @>f'>> \mathcal H'\\
\end{CD} &\qquad\qquad & \begin{CD}
\mathcal H @>g>> \mathcal H\\
@V\Psi(x)VV @VV\Psi(x)V\\
\mathcal H' @>g'>> \mathcal H'\\
\end{CD}\\
\end{array}
\end{equation}
Then for each $x\in C$, the following diagram also commutes.
\begin{equation}
\begin{CD}
\mathcal H @>f\odot g>> \mathcal H\\
@V\Psi(x)VV @VV\Psi(x)V\\
\mathcal H' @>f'\odot g'>> \mathcal H'\\
\end{CD}
\end{equation}
\end{lem}

\begin{proof}
Let $h\in \mathcal H$. We write $\Delta_{\mathcal H}(h):=h_{(1)}\otimes h_{(2)}$ in Sweedler notation, by suppressing summation signs. For $x\in C$,  we verify that
\begin{equation*}
\begin{array}{lr}
\Psi(x)((f\odot g)(h))&\\
=\Psi(x)(f(h_{(1)})g(h_{(2)}))&\\
=\Psi(x_{(1)})(f(h_{(1)}))\Psi(x_{(2)})(g(h_{(2)}))&\mbox{$\qquad$(since $\Psi:C\longrightarrow Hom_K(\mathcal H,\mathcal H')$ is a measuring)}\\
=f'(\Psi(x_{(1)})(h_{(1)}))g'(\Psi(x_{(2)})(h_{(2)})) & \mbox{(applying \eqref{44cda})}\\
=f'(\Psi(x)(h)_{(1)})g'(\Psi(x)(h)_{(2)})&\mbox{$\qquad$(since $\Psi:C\longrightarrow Hom_K(\mathcal H,\mathcal H')$ is a measuring)}\\
=(f'\odot g')(\Psi(x)(h))&\\
\end{array}
\end{equation*}
\end{proof} 

Suppose that $\mathcal H$ is also a graded Hopf algebra, i.e., $\mathcal H=\oplus_{n\geq 0}\mathcal H_n$. Let $\mathcal H_0=K$ and let $f\in End_K(\mathcal H)$ be a $K$-linear map of degree $0$ such that $f(1)=0$. Then, $f^{\odot k}|_{\mathcal H_n}=0$ whenever $k>n$ and we note that the following is a well defined $K$-linear endomorphism of $\mathcal H$ (see \cite[$\S$ 4.5.2.1]{Lod}):
\begin{equation}\label{idem46p}
e^{(1)}(f):=\log (\iota_{\mathcal H}\epsilon_{\mathcal H}+f)=f- \frac{f^{\odot 2}}{2} + ... + (-1)^{i+1}\frac{f^{\odot i}}{i}+...
\end{equation} of degree $0$. Accordingly, one can set
\begin{equation}\label{idem4p}
e^{(i)}(f):=\frac{(e^{(1)}(f))^{\odot i}}{i!} \qquad e^{(i)}_n(f):=e^{(i)}(f)|_{\mathcal H_n}
\end{equation} Since $f$ is of degree $0$, we note that $e^{(i)}_n(f)$ becomes an endmorphism of $\mathcal H_n$.

\smallskip
In particular, let $\mathcal H:=\mathbb T(A)$ be the graded cotensor Hopf algebra of $A$ (see \cite[$\S$ 4.5.4]{Lod}). In other words, $\mathcal H_0=K$ and $\mathcal H_n:=A^{\otimes n}$ for $n\geq 1$. The coproduct $\Delta_{\mathcal H}$ on $\mathcal H=\mathbb T(A)$ is determined by the ``cut product''
\begin{equation}
\Delta_{\mathcal H}(a_1,...,a_n)=\sum_{i=0}^n (a_1,...,a_i)\otimes (a_{i+1},...,a_n) \qquad (a_1,...,a_n)\in \mathcal H_n=A^{\otimes n}
\end{equation} For $n\geq 1$, there is a left action of $S_n$ on $A^{\otimes n}$ given by setting $\tau\cdot (a_1,...,a_n)=(a_{\tau^{-1}(1)},...,a_{\tau^{-1}(n)})$ for $\tau\in S_n$ and $ (a_1,...,a_n)\in A^{\otimes n}$. The multiplication $\mu_{\mathcal H}$ on $\mathcal H=\mathbb T(A)$ is now determined  by the shuffle product
\begin{equation}\label{4.8sh}
\mu_{\mathcal H}((a_1,...,a_p)\otimes (b_{1},...,b_{q})):=\sum_{\sigma\in S_{p,q}}sgn(\sigma)\sigma\cdot (a_1,...,a_p,b_1,...,b_q)\qquad (a_1,...,a_p)\in 
A^{\otimes p}, \textrm{ }(b_1,...,b_q)\in A^{\otimes q}
\end{equation} where $S_{p,q}\subseteq S_{p+q}$ denotes the collection of $(p,q)$-shuffles as in Section 3. 

\begin{lem}\label{Lm42}
Let $C$ be a cocommutative coalgebra and let $\Phi:C\longrightarrow Hom_K(A,A')$ be a measuring between commutative algebras.  Let $\mathcal H:=\mathbb T(A)$ and $\mathcal H':=\mathbb T(A')$ be the respective graded cotensor algebras. Then, there is an induced measuring 
$\mathbb T(\Phi):C\longrightarrow Hom_K(\mathcal H,\mathcal H')$ determined by setting
\begin{equation}\label{4.10ge}
\mathbb T(\Phi)(x)(a_1,...,a_n):=(\Phi(x_{(1)}(a_1),...,\Phi(x_{(n)})(a_n))=(x_{(1)}(a_1),...,x_{(n)}(a_n))\qquad \mathbb T(\Phi)(x)(1):=\epsilon_C(x)
\end{equation} for $x\in C$ and $(a_1,...,a_n)\in A^{\otimes n}$, $n\geq 1$. 
\end{lem}

\begin{proof}
We let $\mu_{\mathcal H}$ and $\mu_{\mathcal H'}$ denote respectively the multiplications on the graded cotensor algebras $\mathbb T(A)$ and 
$\mathbb T(A')$. Because $C$ is cocommutative,  we note that 
\begin{equation}\label{4.11x}
\Phi(x)(\tau\cdot (a_1,...,a_n))= \tau \cdot (x_{(1)}(a_1),...,x_{(n)}(a_n))\qquad x\in C,\tau\in S_n, (a_1,...,a_n)\in A^{\otimes n}
\end{equation} We  consider $(a_1,...,a_p)\in \mathcal H_p=A^{\otimes p}$ and $(b_1,...,b_q)\in \mathcal H_q=A^{\otimes q}$.  
For $x\in C$, we now have
\begin{equation*}
\begin{array}{ll}
\mathbb T(\Phi)(x)(\mu_{\mathcal H}((a_1,...,a_p)\otimes (b_{1},...,b_{q})))&
=\mathbb T(\Phi)(x)\left(\sum_{\sigma\in S_{p,q}}sgn(\sigma)\sigma\cdot (a_1,...,a_p,b_1,...,b_q)\right)\\
&= \sum_{\sigma\in S_{p,q}}sgn(\sigma)\sigma\cdot (x_{(1)}(a_1),...,x_{(p)}(a_p),x_{(p+1)}(b_1),...,x_{(p+q)}(b_q))\\
&=\mu_{\mathcal H'}(\mathbb T(\Phi)(x_{(1)})(a_1,...,a_p)\otimes \mathbb T(\Phi)(x_{(2)})(b_1,...,b_q) )\\
\end{array}
\end{equation*} This proves the result.
\end{proof}

By definition, the counit $\epsilon_{\mathcal H}:\mathcal H\longrightarrow K$ on $\mathcal H=\mathbb T(A)$ is determined by the identity on $\mathcal H_0=K$ and $\epsilon_{\mathcal H}|_{\mathcal H_n}=0$ for $n>0$. The unit on $\mathcal H$ is given by $\iota_{\mathcal H}:K\xrightarrow{id}  \mathcal H_0=K$.
Then $f:=Id - \iota_{\mathcal H}\epsilon_{\mathcal H}\in End_K(\mathcal H)$ satisfies $f(1)=0$ and we set
\begin{equation}\label{lamb1}
e^{(i)}:=e^{(i)}(Id- \iota_{\mathcal H}\epsilon_{\mathcal H}) \qquad e_n^{(i)}:=e^{(i)}(Id- \iota_{\mathcal H}\epsilon_{\mathcal H})|_{\mathcal H_n}
\end{equation} in the sense of \eqref{idem4p}. We know (see for instance, \cite[$\S$ 4.5.3]{Lod}) that $e_n^{(i)}=0$ whenever $i>n$ and that for each $n$, the collection $\{e_n^{(i)}\}_{1\leq i\leq n}$ is a set of mutually orthogonal idempotents whose sum is the identity. Similarly for $\mathcal H'=\mathbb T(A')$, we can consider $f':=Id - \iota_{\mathcal H'}\epsilon_{\mathcal H'}\in End_K(\mathcal H')$ where $\iota_{\mathcal H'}$ and $\epsilon_{\mathcal H'}$ are the unit and counit on $\mathcal H'$ respectively. Accordingly, we set $e'^{(i)}:=e'^{(i)}(Id- \iota_{\mathcal H'}\epsilon_{\mathcal H'})$ and $ e'^{(i)}_n:=e'^{(i)}(Id- \iota_{\mathcal H'}\epsilon_{\mathcal H'})|_{\mathcal H'_n}$.

\begin{lem}\label{Lm43}
Let $C$ be a cocommutative coalgebra and let $\Phi:C\longrightarrow Hom_K(A,A')$ be a measuring between commutative algebras. Let $\mathcal H:=\mathbb T(A)$ and $\mathcal H':=\mathbb T(A')$ be the respective graded cotensor algebras. Then, the following diagram commutes
\begin{equation}\label{413cd}
\begin{CD}
C_n(A)=A\otimes A^{\otimes n}=A\otimes \mathcal H_n @>A\otimes e_n^{(i)}>>  A\otimes \mathcal H_n=A\otimes A^{\otimes n}=C_n(A)\\
@VC_n^\Phi(x)VV @VVC_n^\Phi(x)V\\
C_n(A')=A'\otimes A'^{\otimes n}=A'\otimes \mathcal H'_n @>A'\otimes e'^{(i)}_n>>  A'\otimes \mathcal H'_n=A'\otimes A'^{\otimes n}=C_n(A')\\
\end{CD}
\end{equation}
for each $i,n\geq 1$ and $x\in C$.
\end{lem}
\begin{proof}
Applying Lemma \ref{Lm42}, we see that the measuring  $\Phi:C\longrightarrow Hom_K(A,A')$ induces $\mathbb T(\Phi):C\longrightarrow Hom_K(\mathcal H,\mathcal H')$, which is also a coalgebra measuring. For $f=Id - \iota_{\mathcal H}\epsilon_{\mathcal H}\in End_K(\mathcal H)$ and $f'=Id - \iota_{\mathcal H'}\epsilon_{\mathcal H'}\in End_K(\mathcal H')$, it is immediate that $f'\circ  \mathbb T(\Phi)(x)=\mathbb T(\Phi)(x)\circ f$ for each $x\in C$. Since $f$, $f'$ as well as the maps $\{\Phi(x)\}_{x\in C}$ are all of degree $0$, it now follows from Lemma \ref{Lm4.1} and the definitions in \eqref{idem46p} and \eqref{idem4p} that $e'^{(i)}_n\circ \mathbb T(\Phi)(x)= \mathbb T(\Phi)(x)\circ e^{(i)}_n$ for each $x\in C$.

\smallskip
For $x\in C$ and $(a_0,a_1,...,a_n)\in A\otimes A^{\otimes n}$, we now see that
\begin{equation}
\begin{array}{ll}
((A\otimes e'^{(i)}_n)\circ C_n^\Phi(x))(a_0,a_1,...,a_n)&=\Phi(x_{(1)})(a_0)\otimes ((e'^{(i)}_n\circ \mathbb T(\Phi)(x_{(2)}))(a_1,...,a_n))\\
&=\Phi(x_{(1)})(a_0)\otimes (\mathbb T(\Phi)(x_{(2)})\circ e^{(i)}_n)(a_1,...,a_n))\\
&=(C_n^\Phi(x)\circ (A\otimes e_n^{(i)}))(a_0,a_1,...,a_n)\\
\end{array}
\end{equation} This proves the result. 
\end{proof}

We know (see \cite[$\S$ 4.5.10]{Lod}) that the maps $\{A\otimes e_n^{(i)}:C_n(A)\longrightarrow C_n(A)\}_{1\leq i\leq n}$ induced by the idempotents 
$\{e_n^{(i)}\}_{1\leq i\leq n}$ lead to a direct sum decomposition $C_n(A)=C_n^{(1)}(A)\oplus ... \oplus C_n^{(n)}(A)$ which further splits the Hochschild complex $(C_\bullet^{hoc}(A),b)$ into a direct sum of subcomplexes $C_\bullet^{hoc,(i)}(A)$.  Accordingly, the Hochschild homology admits a direct sum decomposition $HH_n(A)=HH_n^{(1)}(A)\oplus ... \oplus HH_n^{(n)}(A)$ for $n\geq 1$ as mentioned in \eqref{4.1spl}. We also set $e^{(0)}_0:=id$ and $HH_0^{(0)}(A):=HH_0(A)$.  We are now ready to show that this direct sum decomposition of Hochschild homology is compatible with maps induced by coalgebra measurings. 

\begin{Thm}\label{T4.4}
Let $C$ be a cocommutative $K$-coalgebra and let $\Phi:C\longrightarrow Hom_K(A,A')$ be a  measuring between commutative algebras. For  each $x\in C$ and $i,n\geq 0$, we have maps 
\begin{equation}
C_n^{\Phi,(i)}(x): C_n^{(i)}(A)\longrightarrow C_n^{(i)}(A')
\end{equation} induced by restricting the maps $C_n^\Phi(x): C_n(A)\longrightarrow C_n(A')$. 
In particular, we have induced maps
\begin{equation}\label{4.16ind}
HH^{\Phi,(i)}_\bullet(x): HH_\bullet^{(i)}(A)\longrightarrow HH_\bullet^{(i)}(A')
\end{equation}
on direct summands appearing in the $\lambda$-decomposition of Hochschild homology groups. 
\end{Thm}

\begin{proof}
By definition, we know that $C_n^{(i)}(A)=Im(A\otimes e_n^{(i)}:C_n(A)\longrightarrow C_n(A))$.  Applying Lemma \ref{Lm43}, we have the following commutative diagram for any $1\leq i,j,\leq n$ and $x\in C$
\begin{equation}\label{417cd}
\begin{CD}
C_n^{(i)}(A)@>>> C_n(A) @>A\otimes e_n^{(i)}>>  C_n(A)@. \\
@. @VC_n^\Phi(x)VV @VVC_n^\Phi(x)V\\
@. C_n(A')@>A'\otimes e'^{(i)}_n>>C_n(A')@>A'\otimes e'^{(j)}_n>> C_n(A')\\
\end{CD}
\end{equation} where the map $C_n^{(i)}(A)\longrightarrow C_n(A)$ appearing in \eqref{417cd} is the canonical inclusion. Since $\{e'^{(i)}_n\}_{1\leq i\leq n}$ is a collection of mutually orthogonal idempotents whose sum is the identity, it is clear from \eqref{417cd} that the image of $C_n^{(i)}(A)$ under the map $C_n^\Phi(x): C_n(A)\longrightarrow C_n(A')$ lies inside the direct summand $C_n^{(i)}(A')$. The maps in \eqref{4.16ind} now follow directly from this.
\end{proof}

\begin{cor}
Let $C$ be a cocommutative $K$-coalgebra and let $\Phi:C\longrightarrow Hom_K(A,A')$ be a  measuring between commutative algebras.  For $x\in C$ and $i$, $j\geq 1$, the following diagrams commute
\begin{equation}
\begin{CD}
HH_\bullet^{(i)}(A)\otimes HH_\bullet^{(j)}(A) @>\times >> HH_\bullet^{(i+j)}(A)\\
@VHH_\bullet^{\Phi,(i)}(x_{(1)})\otimes HH_\bullet^{\Phi, (j)}(x_{(2)})VV @VVHH_\bullet^{\Phi, (i+j)}(x)V \\
HH_\bullet^{(i)}(A')\otimes HH_\bullet^{(j)}(A') @>\times >> HH_\bullet^{(i+j)}(A')\\
\end{CD}
\end{equation} 
\end{cor}

\begin{proof}
From \cite[Proposition 2.3]{BK1}, we already know that the coalgebra measuring $\Phi$ induces a measuring with respect to the shuffle product on Hochschild homology algebras, i.e., we have
\begin{equation}
HH_{m+n}^\Phi(x)(\tilde a\times \tilde b)=HH_m^\Phi(x_{(1)})(\tilde a)\times HH_n^\Phi(x_{(2)})(\tilde b) 
\end{equation} for $m$, $n\geq 0$ and $\tilde a\in HH_m(A)$, $\tilde b\in HH_n(A)$.  We also know (see \cite[$\S$ 4.5.14]{Lod}) that the shuffle product on Hochschild homology is well behaved with respect to the $\lambda$-decomposition as mentioned in \eqref{4.2shf}. From Theorem \ref{T4.4}, it follows that the maps $HH_\bullet^\Phi(x):HH_\bullet(A)\longrightarrow HH_\bullet(A')$  induced by the measuring are also well behaved with respect to the $\lambda$-decomposition. The result is now clear. 
\end{proof}

We also know that the maps $A\otimes e_n^{(i)}:C_n(A)=A\otimes A^{\otimes n}\longrightarrow A\otimes A^{\otimes n}= C_n(A)$ induce maps
$A\otimes \bar{e}_n^{(i)}:\bar{C}_n(A)=A\otimes \bar{A}^{\otimes n}\longrightarrow A\otimes \bar{A}^{\otimes n}=\bar{C}_n(A)$
 which split  the normalized mixed complex ${\overline{\mathscr BC}}_{\bullet\bullet}(A):=(\bar{C}_\bullet(A),b,B)$ into a direct sum of subcomplexes
${\overline{\mathscr BC}}^{(i)}_{\bullet\bullet}(A)$, $i\geq 0$ (see \cite[$\S$ 4.6.7]{Lod}). Accordingly, the cyclic homology of $A$ has a $\lambda$-decomposition
\begin{equation}\label{4.cspl}
HC_0(A)=HC_0^{(0)}(A)\qquad \qquad HC_n(A)=HC_n^{(1)}(A)\oplus ... \oplus HC_n^{(n)}(A)\qquad n\geq 1
\end{equation}  where $HC_\bullet^{(i)}(A)$ is the total homology of the bicomplex ${\overline{\mathscr BC}}^{(i)}_{\bullet\bullet}(A)$.

\begin{Thm}\label{T4.6}
Let $C$ be a cocommutative $K$-coalgebra and let $\Phi:C\longrightarrow Hom_K(A,A')$ be a  measuring between commutative algebras. For  each $x\in C$ and $i,n\geq 0$, the maps
$
C_n^{\Phi,(i)}(x): C_n^{(i)}(A)\longrightarrow C_n^{(i)}(A')
$  induce maps
\begin{equation}\label{4.21nd}
HC^{\Phi,(i)}_\bullet(x): HC_\bullet^{(i)}(A)\longrightarrow HC_\bullet^{(i)}(A')
\end{equation}
on direct summands appearing in the $\lambda$-decomposition of cyclic homology groups. 
\end{Thm}

\begin{proof}
Let $x\in C$. We know that the maps $A\otimes e_n^{(i)}:C_n(A)=A\otimes A^{\otimes n}\longrightarrow A\otimes A^{\otimes n}= C_n(A)$  (resp. $A'\otimes e'^{(i)}_n:C_n(A')=A'\otimes A'^{\otimes n}\longrightarrow A'\otimes A'^{\otimes n}= C_n(A')$) descend  to $A\otimes \bar{e}_n^{(i)}:\bar{C}_n(A)=A\otimes \bar{A}^{\otimes n}\longrightarrow A\otimes \bar{A}^{\otimes n}=\bar{C}_n(A)$ (resp. $A'\otimes \bar{e}'^{(i)}_n:\bar{C}_n(A')=A'\otimes \bar{A}'^{\otimes n}\longrightarrow A'\otimes \bar{A'}^{\otimes n}=\bar{C}_n(A')$) and the maps $C_n^\Phi(x)$ descend to $\bar{C}^\Phi_n(x)$.  We now consider the following two diagrams
\begin{equation}\label{422cd2}
\begin{array}{ccc}
\begin{CD}
C_n(A)=A\otimes A^{\otimes n}@>A\otimes e_n^{(i)}>>  A\otimes A^{\otimes n}=C_n(A)\\
@VC_n^\Phi(x)VV @VVC_n^\Phi(x)V\\
C_n(A')=A'\otimes A'^{\otimes n} @>A'\otimes e'^{(i)}_n>>=A'\otimes A'^{\otimes n}=C_n(A')\\
\end{CD} &\Rightarrow & \begin{CD}
\bar{C}_n(A)=A\otimes \bar{A}^{\otimes n}@>A\otimes \bar{e}_n^{(i)}>>  A\otimes \bar{A}^{\otimes n}=\bar{C}_n(A)\\
@V\bar{C}_n^\Phi(x)VV @VV\bar{C}_n^\Phi(x)V\\
\bar{C}_n(A')=A'\otimes \bar{A}'^{\otimes n} @>A'\otimes \bar{e}'^{(i)}_n>>=A'\otimes \bar{A}'^{\otimes n}=\bar{C}_n(A')\\
\end{CD}  \\
\end{array}
\end{equation} The left hand side square in \eqref{422cd2} commutes by Lemma \ref{Lm43}. Since the maps $C_n(A)\longrightarrow \bar{C}_n(A)$ are epimorphisms, it follows that the right hand square in \eqref{422cd2} also commutes. From this, it follows that for each $x\in C$, we have induced maps
\begin{equation}
{\overline{\mathscr BC}}^{\Phi,(i)}_{\bullet\bullet}(x):{\overline{\mathscr BC}}^{(i)}_{\bullet\bullet}(A)\longrightarrow {\overline{\mathscr BC}}^{(i)}_{\bullet\bullet}(A')
\end{equation} on the direct summands of the normalized mixed complexes computing cyclic homology. The result is now clear. 
\end{proof} 

We know (see \cite[$\S$ 4.6.9]{Lod}) that the bicomplexes ${\overline{\mathscr BC}}^{(i)}_{\bullet\bullet}(A)$ and the normalized Hochschild complex $\bar{C}_\bullet^{hoc,(i)}(A)$ fit into a short exact sequence $0\longrightarrow \bar{C}_\bullet^{hoc,(i)}(A)\longrightarrow {\overline{\mathscr BC}}^{(i)}_{\bullet\bullet}(A)\longrightarrow {\overline{\mathscr BC}}^{(i-1)}_{\bullet\bullet}(A)[2]\longrightarrow 0$. Accordingly, we have the following result. 

\begin{cor}\label{C4.7}
Let $C$ be a cocommutative $K$-coalgebra and let $\Phi:C\longrightarrow Hom_K(A,A')$ be a  measuring between commutative algebras. For each $x\in C$, the maps $HH_\bullet^{\Phi, (i)}(x)$ and 
$HC_\bullet^{\Phi,(i)}(x)$ fit into the following commutative diagram with respect to the Connes periodicity sequences
\begin{equation}
\begin{CD}
\dots @>>> HH_n^{(i)}(A) @>I>> HC_n^{(i)}(A) @>S>> HC_{n-2}^{(i-1)}(A) @>B>> HH^{(i)}_{n-1}(A) @>>> \dots \\
@. @VHH^{\Phi,(i)}_n(x)VV @VHC^{\Phi,(i)}_n(x)VV @VHC^{\Phi,(i-1)}_{n-2}(x)VV @VHH^{\Phi,(i)}_{n-1}(x)VV @. \\
\dots @>>> HH_n^{(i)}(A') @>I>> HC_n^{(i)}(A') @>S>> HC_{n-2}^{(i-1)}(A') @>B>> HH^{(i)}_{n-1}(A') @>>> \dots \\
\end{CD}
\end{equation} 
\end{cor} 

\section{Measurings and maps in Lie algebra homology and Leibniz homology}

In this section, we relate maps on Lie algebra homology and Leibniz homology induced by measurings to maps on cyclic homology and Hochschild homology. For a  $K$-algebra $A$ and $r\geq 1$, let $M_r(A)$ denote the algebra of $(r\times r)$-matrices. We let $gl_r(A)$ denote the Lie algebra obtained from $M_r(A)$ by setting the Lie bracket to be $[\alpha,\beta]:=\alpha\beta-\beta\alpha$ for matrices $\alpha$, $\beta\in M_r(A)$.  By considering the canonical inclusions $M_r(A)\hookrightarrow M_{r+1}(A)$ obtained by adding $0$ entries to the matrices, we have the direct limit algebra $M(A)$. Similarly, we denote by $gl(A)$ the Lie algebra obtained from $M(A)$ by taking the Lie bracket to be $[\alpha,\beta]:=\alpha\beta-\beta\alpha$ for $\alpha$, $\beta\in M(A)$

\smallskip
We recall (see for instance, \cite{Lod3}) that a Leibniz algebra $\mathfrak g$ over $K$ consists of a $K$-vector space $\mathfrak g$ along with a $K$-bilinear map
$[\_\_,\_\_]:\mathfrak g\times \mathfrak g\longrightarrow \mathfrak g$ which satisfies the relation
\begin{equation}
[x,[y,z]]-[[x,y],z]+[[x,z],y]=0\qquad \mbox{for all}\quad x,y,z\in \mathfrak g
\end{equation} In particular, any Lie algebra is also a Leibniz algebra for the same bracket operation. In \cite[$\S$ 3]{BK1}, we have already considered measurings between Lie algebras. We now broaden this definition to include Leibniz algebras.

\begin{defn}
\label{D5.1} Let $(\mathfrak g,[\_\_\,\_\_])$ and $(\mathfrak g',[\_\_,\_\_]')$ be Lie algebras (resp. Leibniz algebras). A coalgebra measuring $(C,\Psi)$ of Lie algebras (resp. Leibniz algebras) from $\mathfrak g$ to $\mathfrak g'$ consists of a cocommutative $K$-coalgebra $C$ and a $K$-linear map $\Psi:C\longrightarrow Hom_K(\mathfrak g,\mathfrak g')$ which satisfies
\begin{equation}
\Psi(x)([\alpha,\beta])=[\Psi(x_{(1)})(\alpha),\Psi(x_{(2)})(\beta)]'\qquad \alpha,\beta\in \mathfrak g, x\in C
\end{equation} where the coproduct $\Delta_C$ on $C$ is written as $\Delta_C(x)=x_{(1)}\otimes x_{(2)}$ in Sweedler notation by suppressing summation signs for each $x\in C$.
\end{defn}

For a Lie algebra $\mathfrak g$, we know that the classical Chevalley-Eilenberg complex $CE_\bullet(\mathfrak g)$ is given by setting $CE_n(\mathfrak g):=\wedge^n\mathfrak g$ for 
$n\geq 0$ along with   differentials  given by
\begin{equation}\label{cediff}
d_{CE}:CE_n(\mathfrak g)\longrightarrow CE_{n-1}(\mathfrak g)\qquad d_{CE}(\alpha_1\wedge ...\wedge \alpha_n):=\underset{1\leq i<j\leq n}{\sum}(-1)^{i+j+1} [\alpha_i,\alpha_j]
\wedge \alpha_1\wedge ...\wedge \hat{\alpha}_i\wedge ...\wedge \hat{\alpha}_j\wedge ... \wedge \alpha_n
\end{equation} where $\hat{\alpha}_i$ in \eqref{cediff} refers to $\alpha_i$ removed from the sequence. The Lie algebra homology of $\mathfrak g$, denoted by $H_\bullet(\mathfrak g)$ is the homology of the  Chevalley-Eilenberg complex $CE_\bullet(\mathfrak g)$.

\smallskip 
Let $\Psi:C\longrightarrow Hom_K(\mathfrak g,\mathfrak g')$ be a coalgebra measuring between Lie algebras. For each $x\in C$, we showed in \cite[Proposition 3.3]{BK1} that the maps
\begin{equation}
CE^\Psi_n(x):CE_n(\mathfrak g)\longrightarrow CE_n(\mathfrak g') \qquad \alpha_1\wedge ... \wedge \alpha_n\mapsto x( \alpha_1\wedge ... \wedge \alpha_n):= x_{(1)}(\alpha_1)\wedge ... \wedge x_{(n)}(\alpha_n)
\end{equation} determine a map between Chevalley-Eilenberg complexes. This induces maps $H^\Psi_\bullet(x):H_\bullet(\mathfrak g)\longrightarrow H_\bullet(\mathfrak g')$ on Lie algebra homology for each $x\in C$. 

\smallskip
For any Lie algebra $\mathfrak g$, we have the diagonal map $\delta_{\mathfrak g}:\mathfrak g\longrightarrow \mathfrak g\oplus \mathfrak g$ given by $\alpha\mapsto (\alpha,\alpha)$. To distinguish between the two components, we  will often write $\delta_{\mathfrak g}(\alpha)=(\alpha,\alpha)=(\alpha^{(1)},\alpha^{(2)})$. Then, $\delta_{\mathfrak g}$ is a map of Lie algebras and by using the canonical isomorphisms  $\wedge^n(\mathfrak g\oplus \mathfrak g)=\underset{p+q=n}{\sum}\wedge^p\mathfrak g\otimes \wedge^q\mathfrak g$, it is well known (see, for instance, \cite[$\S$ 10.1.5]{Lod}) that we have a coproduct  
\begin{equation}\label{diagco}
\Delta_{\mathfrak g}:H_\bullet(\mathfrak g)\longrightarrow H_\bullet(\mathfrak g)\otimes H_\bullet(\mathfrak g)
\end{equation} on Lie algebra homology. Our first result in this section is that the maps on Lie algebra homologies induced by a coalgebra measuring are compatible with this coproduct. 

\begin{thm}\label{P5.1}   Let $C$ be a cocommutative coalgebra and let $\Psi:C\longrightarrow Hom_K(\mathfrak g,\mathfrak g')$  be a measuring of Lie algebras from $\mathfrak g$ to $\mathfrak g'$. Then for each $x\in C$, the following diagram commutes
\begin{equation}
\begin{CD}
H_\bullet(\mathfrak g) @>\Delta_{\mathfrak g}>> H_\bullet(\mathfrak g) \otimes H_\bullet(\mathfrak g) \\
@VH^\Psi_\bullet(x)VV @VVH^\Psi_\bullet(x_{(1)})\otimes H^\Psi_\bullet(x_{(2)})V\\
H_\bullet(\mathfrak g') @>\Delta_{\mathfrak g'}>> H_\bullet(\mathfrak g') \otimes H_\bullet(\mathfrak g') \\
\end{CD}
\end{equation}
where we write the coproduct $\Delta_C(x):=x_{(1)}\otimes x_{(2)}$ by suppressing the summation signs. 
\end{thm}

\begin{proof}
We take $x\in C$ and $\alpha_1\wedge ... \wedge \alpha_n\in CE_n(\mathfrak g)$. By definition, we have $\wedge^n\delta_{\mathfrak g}(\alpha_1\wedge ... \wedge \alpha_n)=((\alpha_1^{(1)},\alpha_1^{(2)})\wedge ... \wedge (\alpha_n^{(1)},\alpha_n^{(2)}))$. We consider a term $(\alpha_1^{(i_1)},...,\alpha_n^{(i_n)})$ appearing in the expansion of $\wedge^n\delta_{\mathfrak g}(\alpha_1\wedge ... \wedge \alpha_n)$ where $p$ of the terms are in the first component and $q$ of the terms are in the second component. In particular, we have $p+q=n$. Let $\sigma$ be the permutation that rearranges $(\alpha_1^{(i_1)},...,\alpha_n^{(i_n)})$ such that the first $p$ terms are in the first component 
and the remaining $q$ terms are in the second component. Then, by definition of the canonical maps $\wedge^n(\mathfrak g\oplus \mathfrak g)\longrightarrow \wedge^p\mathfrak g\otimes \wedge^q\mathfrak g$ that together determine the isomorphism $\wedge^n(\mathfrak g\oplus\mathfrak g)\overset{\cong} {\longrightarrow} \underset{p+q=n}{\bigoplus}\wedge^p\mathfrak g\otimes \wedge^q\mathfrak g$, the term $(\alpha_1^{(i_1)},...,\alpha_n^{(i_n)})$ appearing in the expansion of $\wedge^n\delta_{\mathfrak g}(\alpha_1\wedge ... \wedge \alpha_n)$  maps to $sgn(\sigma)(\alpha_{\sigma^{-1}(1)}\wedge ... \wedge \alpha_{\sigma^{-1}(p)})\otimes (\alpha_{\sigma^{-1}(p+1)}\wedge ...\wedge \alpha_{\sigma^{-1}(p+q)})\in \wedge^p\mathfrak g\otimes \wedge^q\mathfrak g$. We note that
\begin{equation}\label{5.7yu}
\begin{array}{l}
sgn(\sigma)x_{(1)}(\alpha_{\sigma^{-1}(1)}\wedge ... \wedge \alpha_{\sigma^{-1}(p)})\otimes x_{(2)}(\alpha_{\sigma^{-1}(p+1)}\wedge ...\wedge \alpha_{\sigma^{-1}(p+q)})\\=sgn(\sigma)(x_{(1)}(\alpha_{\sigma^{-1}(1)})\wedge ... \wedge x_{(p)}(\alpha_{\sigma^{-1}(p)}))\otimes (x_{(p+1)}(\alpha_{\sigma^{-1}(p+1)})\wedge ...\wedge x_{(n)}(\alpha_{\sigma^{-1}(p+q)}))\\
\end{array}
\end{equation} Since $C$ is cocommutative, it is clear that the term in \eqref{5.7yu} equates to the term  
\begin{equation} sgn(\sigma)(x_{\sigma^{-1}(1)}(\alpha_{\sigma^{-1}(1)})\wedge ... \wedge x_{\sigma^{-1}(p)}(\alpha_{\sigma^{-1}(p)}))\otimes (x_{\sigma^{-1}(p+1)}(\alpha_{\sigma^{-1}(p+1)})\wedge ...\wedge x_{\sigma^{-1}(p+q)}(\alpha_{\sigma^{-1}(p+q)}))\in \wedge^p\mathfrak g'\otimes \wedge^q\mathfrak g'
\end{equation} appearing in the expansion of $\wedge^n\delta_{\mathfrak g'}(x(\alpha_1\wedge ... \wedge \alpha_n))=\wedge^n\delta_{\mathfrak g'}(x_{(1)}(\alpha_1)\wedge ... \wedge x_{(n)}(\alpha_n))$. By comparing term by term, the result is now clear. 
\end{proof}

We now set $\mathfrak g:=gl(A)$ to be the Lie algebra of matrices over $A$. We denote by $\Delta:=\Delta_{\mathfrak g}:H_\bullet(gl(A))\longrightarrow H_\bullet(gl(A))\otimes H_\bullet(gl(A))$ the coproduct on 
$H_\bullet(\mathfrak g)=H_\bullet(gl(A))$ induced by the diagonal map as in \eqref{diagco}. Additionally, we know (see \cite[$\S$ 10.2]{Lod}) that $H_\bullet(gl(A))$ is a graded Hopf algebra and that we have an isomorphism
\begin{equation}\label{5.9iso}
Prim(H_\bullet(gl(A))):=\{\mbox{$\alpha\in H_\bullet(\mathfrak g)=H_\bullet(gl(A))$ $\vert$ $\Delta(\alpha)=1\otimes \alpha+\alpha\otimes 1$}\}\cong HC_{\bullet-1}(A)
\end{equation}
between the primitive part $Prim(H_\bullet(gl(A)))$ of the Hopf algebra $H_\bullet(gl(A))$ and the cyclic homology (shifted by degree $1$) of $A$. Similarly, we set $\mathfrak g'$ to be the Lie algebra 
$\mathfrak g':=gl(A')$ and $\Delta':=\Delta_{\mathfrak g'}$ to be the coproduct on $H_\bullet(gl(A'))$. 

\smallskip
Now let $C$ be a cocommutative coalgebra and let $\Phi:C\longrightarrow Hom_K(A,A')$ be a measuring between  algebras. As noted in \cite[$\S$ 3]{BK1}, we can verify  that this induces a measuring of Lie algebras from $gl(A)$ to $gl(A')$ (resp. from each $gl_r(A)$ to $gl_r(A')$), which we denote by 
$gl(\Phi):C\longrightarrow Hom_K(gl(A),gl(A'))$ (resp. $gl_r(\Phi):C\longrightarrow Hom_K(gl_r(A),gl_r(A'))$).  Accordingly, we have induced maps
\begin{equation}\label{5.10j}
H_\bullet^{gl(\Phi)}(x):H_\bullet(gl(A))\longrightarrow H_\bullet(gl(A'))\qquad H_\bullet^{gl_r(\Phi)}(x):H_\bullet(gl_r(A))\longrightarrow H_\bullet(gl_r(A')) \qquad x\in C
\end{equation} From Proposition \ref{P5.1}, it is clear that the maps in \eqref{5.10j} restrict to maps between the   primitive parts of the Hopf algebras $H_\bullet(gl(A))$ and
$H_\bullet(gl(A'))$, which we continue to denote by $H_\bullet^{gl(\Phi)}(x):Prim(H_\bullet(gl(A)))\longrightarrow Prim(H_\bullet(gl(A')))$. Our next aim is to show that the isomorphism in 
\eqref{5.9iso} is compatible with the maps induced by the measuring $\Phi$.  From \cite[$\S$ 10.2.3]{Lod}, we know that 
\begin{equation}\label{CE5.10}
\theta^A_{n,r}: \wedge^{n+1}gl_r(A)\longrightarrow \tilde C_n(M_r(A))\qquad \theta^A_{n,r}(\alpha_0\wedge \alpha_1\wedge ... \wedge \alpha_n):=\underset{\sigma\in S_n}{\sum}
sgn(\sigma) (\alpha_0,\alpha_{\sigma(1)},...,\alpha_{\sigma(n)})
\end{equation} determines a morphism from the Chevalley-Eilenberg complex of the Lie algebra $gl_r(A)$ to the complex $\tilde C_\bullet(M_r(A))$ (see Remark \ref{R3.2} for notation) computing the cyclic homology of $M_r(A)$. We now have the following result. 

\begin{lem}\label{L5.2}
For each $x\in C$ and $r\geq 1$, the following diagram commutes
\begin{equation}\label{comd5.12}
\begin{CD}
CE_{n+1}(gl_r(A))=\wedge^{n+1}gl_r(A) @>\theta^A_{n,r}>>  \tilde C_n(M_r(A))\\
@VCE_{n+1}^{gl_r(\Phi)}(x)VV @VV\tilde C_n^{M_r(\Phi)}(x)V \\
CE_{n+1}(gl_r(A'))=\wedge^{n+1}gl_r(A') @>\theta^{A'}_{n,r}>>  \tilde C_n(M_r(A'))\\
\end{CD}
\end{equation} where $M_r(\Phi):C\longrightarrow Hom_K(M_r(A),M_r(A'))$ denotes the measuring of matrix algebras induced by $\Phi: C\longrightarrow 
 Hom_K(A,A')$.
\end{lem}

\begin{proof}
We consider $\alpha_0\wedge \alpha_1\wedge ... \wedge \alpha_n\in \wedge^{n+1}gl_r(A)$ and $x\in C$. Since $C$ is cocommutative, we observe that
\begin{equation}
\begin{array}{ll}
(\tilde C_n^{M_r(\Phi)}(x)\circ \theta^A_{n,r})(\alpha_0\wedge \alpha_1\wedge ... \wedge \alpha_n)&=\tilde C_n^{M_r(\Phi)}(x)\left(\underset{\sigma\in S_n}{\sum}
sgn(\sigma) (\alpha_0,\alpha_{\sigma(1)},...,\alpha_{\sigma(n)})\right)\\
& =\underset{\sigma\in S_n}{\sum}
sgn(\sigma) (x_{(1)}(\alpha_0),x_{(2)}(\alpha_{\sigma(1)}),...,x_{(n+1)}(\alpha_{\sigma(n)})) \\
&= \underset{\sigma\in S_n}{\sum}
sgn(\sigma) (x_{(1)}(\alpha_0),x_{\sigma(1)+1}(\alpha_{\sigma(1)}),...,x_{\sigma(n)+1}(\alpha_{\sigma(n)}))\\
&=( \theta^{A'}_{n,r}\circ CE_{n+1}^{gl_r(\Phi)}(x))(\alpha_0\wedge \alpha_1\wedge ... \wedge \alpha_n)\\
\end{array}
\end{equation} This proves the result. 
\end{proof}

By \cite[$\S$ 1.2.1, $\S$ 2.2.10]{Lod}, we know that the quasi-isomorphism of complexes leading to the Morita invariance $HC_\bullet(M_r(A))\cong HC_\bullet(A)$ of cyclic homology is induced by the generalized trace map
\begin{equation}\label{gentr5}
tr^A_{n,r}: \tilde C_n(M_r(A))\longrightarrow \tilde C_n(A)\qquad (\alpha\otimes \beta\otimes ...\otimes \eta)\mapsto \sum \alpha_{i_0i_1}\otimes 
\beta_{i_1i_2}\otimes ...\otimes \eta_{i_ni_0}
\end{equation} where the sum is taken over all possible tuples $(i_0,...,i_n)$ involving the entries of the matrices $\alpha$, $\beta$,..., $\eta$. We can directly verify that the maps in 
\eqref{gentr5} are compatible with a measuring $\Phi:C\longrightarrow Hom_K(A,A')$ of algebras, i.e., $\tilde C_n^{\Phi}(x)\circ tr^A_{n,r}= tr^{A'}_{n,r}\circ\tilde C_n^{M_r(\Phi)}(x)$ for each $x\in C$. We are now ready to show that the isomorphisms in \eqref{5.9iso} are compatible with measurings.

\begin{Thm}\label{T5.4}
Let $C$ be a cocommutative $K$-coalgebra and let $\Phi:C\longrightarrow Hom_K(A,A')$ be a  measuring between  algebras. For each $x\in C$, the following diagram commutes
\begin{equation}
\begin{CD}
Prim(H_\bullet(gl(A)))@>\cong>> HC_{\bullet-1}(A)\\
@VH_\bullet^{gl(\Phi)}(x)VV @VVHC_{\bullet-1}^\Phi(x) V \\
Prim(H_\bullet(gl(A')))@>\cong>> HC_{\bullet-1}(A')\\
\end{CD}
\end{equation} where the isomorphisms in the horizontal arrows are as given by \eqref{5.9iso}.
\end{Thm}
\begin{proof}
From Lemma \ref{L5.2} and the discussion above, it follows that  we have commutative diagrams for $x\in C$ and $r\geq 1$
\begin{equation}\label{comd5.16}
\begin{CD}
CE_{n+1}(gl_r(A))=\wedge^{n+1}gl_r(A) @>\theta^A_{n,r}>>  \tilde C_n(M_r(A))@>tr^A_{n,r}>>\tilde C_n(A) \\
@VCE_{n+1}^{gl_r(\Phi)}(x)VV @VV\tilde C_n^{M_r(\Phi)}(x)V  @VV\tilde C_n^\Phi(x)V \\
CE_{n+1}(gl_r(A'))=\wedge^{n+1}gl_r(A') @>\theta^{A'}_{n,r}>>  \tilde C_n(M_r(A'))@>tr_{n,r}^{A'}>> \tilde C_n(A')\\
\end{CD}
\end{equation}  It is clear that the commutativity in \eqref{comd5.16} passes to the colimit as $r\longrightarrow \infty$. We know (see \cite[$\S$ 10.2.4]{Lod}) that the map $H_\bullet(gl(A))\longrightarrow HC_{\bullet-1}(A) $ which restricts to the isomorphism in \eqref{5.9iso} is obtained from the colimit of the maps
$tr^A_{n,r}\circ \theta^A_{n,r}$ as $r\longrightarrow \infty$. Further, we have noted before that the map $H_\bullet^{gl(\Phi)}(x):H_\bullet(gl(A))\longrightarrow H_\bullet(gl(A'))$ restricts to the primitive parts of the Hopf algebras $H_\bullet(gl(A))$ and $H_\bullet(gl(A'))$. The result is now clear. 
\end{proof}

In the rest of this section, our aim is to prove a result similar to Theorem \ref{T5.4} involving Hochschild homology and Leibniz homology. Accordingly, for the rest of this section, let $(\mathfrak g,[\_\_\,\_\_])$ and $(\mathfrak g',[\_\_,\_\_]')$ be   Leibniz algebras. We recall that the Leibniz homology $HL_\bullet(\mathfrak g)$ of $\mathfrak g$  is computed by means of the complex $CL_\bullet(\mathfrak g)$ (see Loday \cite[$\S$ 6.3]{Lod3}) given by setting $CL_n(\mathfrak g):=\mathfrak g^{\otimes n}$ along with differentials given by
\begin{equation}\label{cl5.17}
d_{CL}:CL_n(\mathfrak g)=\mathfrak g^{\otimes n}\longrightarrow \mathfrak g^{\otimes n-1}=CL_{n-1}(\mathfrak g)\qquad d_{CL}(\alpha_1, ...,\alpha_n):=\underset{1\leq i<j\leq n}{\sum}(-1)^{j}  ( \alpha_1, ...,\alpha_{i-1}, [\alpha_i,\alpha_j], ...,\hat{\alpha}_j, ..., \alpha_n)
\end{equation} 
For graded vector spaces $M=\oplus_{i\geq 0}M_i$, $N=\oplus_{i\geq 0}N_i$ with $M_0=N_0=K$, we recall (see  \cite[$\S$ 6.7]{Lod3}) the tensor product $M\boxtimes N=\oplus_{i\geq 0}(M\boxtimes N)_i$ which is given in degree $n$ by the direct  sum of the $2^n$ components of the form $X_{i_1}\otimes 
Y_{i_2}\otimes X_{i_3}\otimes Y_{i_4}\otimes ...$ where $X=M$ and $Y=N$, or $X=N$ and $Y=M$ with $i_j\geq 1$ and $\sum_ji_j=n$.  As with Lie algebra homology, the diagonal map $\mathfrak g\longrightarrow \mathfrak g\oplus \mathfrak g$ of Leibniz algebras leads to a coproduct $\Delta_{\mathfrak g}:HL_\bullet(\mathfrak g)\longrightarrow HL_\bullet(\mathfrak g)\boxtimes HL_\bullet(\mathfrak g)$ on Leibniz homology (see Loday \cite[$\S$ 6.9]{Lod3}). We first need the following general result for measurings of Leibniz algebras.

\begin{thm}\label{P5.5}
Let $(\mathfrak g,[\_\_\,\_\_])$ and $(\mathfrak g',[\_\_,\_\_]')$ be   Leibniz algebras. Let $C$ be a cocommutative coalgebra and let 
$\Psi: C\longrightarrow Hom_K(\mathfrak g,\mathfrak g')$ be a measuring between Leibniz algebras. For each $x\in C$, the maps
\begin{equation}
CL_n^\Psi(x):CL_n(\mathfrak g)\longrightarrow CL_n(\mathfrak g')\qquad (\alpha_1,...,\alpha_n)\mapsto x(\alpha_1,...,\alpha_n):= (x_{(1)}(\alpha_1), ... , x_{(n)}(\alpha_n))
\end{equation} induce a morphism of complexes from $CL_\bullet(\mathfrak g)$ to $CL_\bullet(\mathfrak g')$. In particular, there is an induced map
$HL^\Psi_\bullet(x):HL_\bullet(\mathfrak g)\longrightarrow HL_\bullet(\mathfrak g')$ on Leibniz homologies for each $x\in C$. Additionally, for each $x\in C$, the following diagram commutes
\begin{equation}\label{5.19cdr}
\begin{CD}
HL_\bullet(\mathfrak g) @>\Delta_{\mathfrak g}>> HL_\bullet(\mathfrak g) \boxtimes HL_\bullet(\mathfrak g) \\
@VHL^\Psi_\bullet(x)VV @VVHL^\Psi_\bullet(x_{(1)})\boxtimes HL^\Psi_\bullet(x_{(2)})V\\
HL_\bullet(\mathfrak g') @>\Delta_{\mathfrak g'}>> HL_\bullet(\mathfrak g') \boxtimes HL_\bullet(\mathfrak g') \\
\end{CD}
\end{equation}
where we write the coproduct $\Delta_C(x):=x_{(1)}\otimes x_{(2)}$ by suppressing the summation signs. 
\end{thm}

\begin{proof} We consider $(\alpha_1,...,\alpha_n)\in\mathfrak g^{\otimes n}$. Since $C$ is cocommutative, we see that for $x\in C$, we have
\begin{equation}
\begin{array}{ll}
x(d_{CL}(\alpha_1,...,\alpha_n))&=x\left(\underset{1\leq i<j\leq n}{\sum}(-1)^{j}  ( \alpha_1, ...,[\alpha_i,\alpha_j], \alpha_{i+1},...,\hat{\alpha}_j, ..., \alpha_n)\right)\\
&=  \underset{1\leq i<j\leq n}{\sum}(-1)^{j}  ( x_{(1)}(\alpha_1), ...,x_{(i)}([\alpha_i,\alpha_j]), x_{(i+1)}(\alpha_{i+1}), ...,\hat{\alpha}_j, ..., x_{(n-1)}(\alpha_n))\\
&=  \underset{1\leq i<j\leq n}{\sum}(-1)^{j}  ( x_{(1)}(\alpha_1), ...,[x_{(i)}(\alpha_i),x_{(i+1)}(\alpha_j)], x_{(i+2)}(\alpha_{i+1}), ...,\hat{\alpha}_j, ..., x_{(n)}(\alpha_n))\\
&=  \underset{1\leq i<j\leq n}{\sum}(-1)^{j}  ( x_{(1)}(\alpha_1), ...,[x_{(i)}(\alpha_i),x_{(j)}(\alpha_j)], x_{(i+1)}(\alpha_{i+1}), ...,\hat{\alpha}_j, ..., x_{(n)}(\alpha_n))\\
&=d_{CL}(x_{(1)}(\alpha_1), ... , x_{(n)}(\alpha_n))\\
\end{array}
\end{equation}
It follows that each $CL_\bullet^\Psi(x):CL_\bullet(\mathfrak g)\longrightarrow CL_\bullet(\mathfrak g')$ is a morphism of complexes, inducing maps on homology groups. The commutativity of \eqref{5.19cdr} can be verified in a manner similar to Proposition \ref{P5.1}.
\end{proof}

We now set $\mathfrak g:=gl(A)$, $\mathfrak g':=gl(A')$ and $\Psi:=gl(\Phi):C\longrightarrow Hom_K(gl(A),gl(A'))$. It is clear that $gl(\Phi)$ is also a measuring of Leibniz algebras.  We know that the Leibniz homology $HL_\bullet(gl(A))$ of $gl(A)$  forms a Hopf algebra (see \cite[$\S$ III.2]{Cuv}). By the commutativity of \eqref{5.19cdr}, it follows that the maps
$HL^{gl(\Phi)}_\bullet(x):HL_\bullet(gl(A))\longrightarrow HL_\bullet(gl(A'))$ for $x\in C$ restrict to $HL^{gl(\Phi)}_\bullet(x):Prim(HL_\bullet(gl(A)))\longrightarrow Prim(HL_\bullet(gl(A')))$ on the respective primitive parts of these Hopf algebras. We will now show that the isomorphisms (see \cite[Th\'{e}or\`{e}me III.3]{Cuv1}) 
\begin{equation}\label{5.9isob}
Prim(HL_\bullet(gl(A)))\cong HH_{\bullet-1}(A)
\end{equation}
are compatible with measurings. This will require several steps. First, we recall (see \cite[$\S$ 10.2.10]{Lod})  that the complex $CL_\bullet(gl_r(A))$ has a quotient $\overline{CL}_\bullet(gl(A))$ determined by epimorphisms 
\begin{equation}\label{5.2f1}
CL_n(gl_r(A)):=(gl_r(A))^{\otimes n}\longrightarrow \overline{CL}_n(gl_r(A))=(gl_r(A))^{\otimes n}_{gl_r(K)}
\end{equation} whose target is the coinvariant $(gl_r(A))^{\otimes n}_{gl_r(K)}\cong ((gl_r(K))^{\otimes n}\otimes A^{\otimes n})_{gl_r(K)}$  produced by combining the adjoint action of  
$gl_r(K)$ on $(gl_r(K))^{\otimes n}$  along with the trivial action on $A^{\otimes n}$. Accordingly, the maps in \eqref{5.2f1} are compatible with measurings and for each $x\in C$ we have a commutative diagram 
\begin{equation}\label{5.2f2}
\begin{CD}
CL_n(gl_r(A))=(gl_r(A))^{\otimes n}@>>> \overline{CL}_n(gl_r(A))=(gl_r(A))^{\otimes n}_{gl_r(K)}\\
@VCL_n^{\Psi_r}(x)VV @VV\overline{CL}_n^{\Psi_r}(x)V \\
CL_n(gl_r(A'))=(gl_r(A'))^{\otimes n}@>>> \overline{CL}_n(gl_r(A'))=(gl_r(A'))^{\otimes n}_{gl_r(K)}\\
\end{CD}
\end{equation} where $\Psi_r:=gl_r(\Phi):C\longrightarrow Hom_K(gl_r(A),gl_r(A'))$ is the measuring of Leibniz algebras induced by $\Phi$. Since the maps in \eqref{5.2f1} are epimorphisms, it follows that the maps $\overline{CL}_n^{\Psi_r}(x)$ in \eqref{5.2f2} determine a morphism  $\overline{CL}_\bullet^{\Psi_r}(x):\overline{CL}_\bullet(gl_r(A))\longrightarrow \overline{CL}(gl_r(A'))$ of complexes. We also know (see \cite[$\S$ 10.6.6]{Lod}) that the horizontal arrows in \eqref{5.2f2} induce quasi-isomorphisms of complexes. It follows that the Leibniz homology $HL_\bullet(gl(A))$ can be computed by taking the colimit of the complexes $ \overline{CL}_\bullet(gl_r(A))$ as $r\longrightarrow \infty$. 

\smallskip
We recall now (see \cite[$\S$ 10.6.7]{Lod}) the isomorphisms (for $r\geq n$)
\begin{equation}\label{iso5.g1}
\Theta^A_n:  \overline{CL}_n(gl_r(A))=(gl_r(A))^{\otimes n}_{gl_r(K)}=((gl_r(K))^{\otimes n}\otimes A^{\otimes n})_{gl_r(K)} \xrightarrow{\qquad\cong \qquad} K[S_n]\otimes A^{\otimes n}
\end{equation} This leads to a complex $L_\bullet(A)$ defined by setting $L_n(A):=K[S_{n+1}]\otimes A^{\otimes n+1}$ whose differential is determined by the isomorphisms in
\eqref{iso5.g1}. It follows that $HL_\bullet(gl(A))\cong H_{\bullet-1}(L_\bullet(A))$. From \eqref{5.2f2} and the isomorphisms in \eqref{iso5.g1},  we have induced maps $H_{n-1}^{\Phi}(x): H_{n-1}(L_\bullet(A))\longrightarrow  H_{n-1}(L_\bullet(A'))$ that fit into the commutative diagram 
\begin{equation}\label{step5.1g}
\begin{CD}
HL_n(gl(A)) @>\cong>> H_{n-1}(L_\bullet(A))\\
@VHL_n^{\Psi}(x)VV @VVH_{n-1}^{\Phi}(x)V\\
HL_n(gl(A')) @>\cong>> H_{n-1}(L_\bullet(A'))\\
\end{CD}
\end{equation}  We now let $U_n$ denote the conjugacy class of  the cycle $(1,2,...,n)$ in the symmetric group $S_n$. For any $\sigma\in U_n$,  let $\omega(\sigma)$ be the unique element in $S_n$ such that $\sigma=\omega(\sigma)(1,2,...,n)\omega(\sigma)^{-1}$ and $\omega(\sigma)(1)=1$.  As before, we know that $S_n$ has a left action on $A^{\otimes n}$ given by $\tau\cdot (a_1,...,a_n):=(a_{\tau^{-1}(1)},...,a_{\tau^{-1}(n)})$ for $\tau
\in S_n$ and $(a_1,...,a_n)\in A^{\otimes n}$. 

\smallskip We will now recall the  presimplicial module $V_\bullet(A)$ (see \cite{Cuv1}) defined by setting $V_n(A):=K[U_{n+1}]\otimes A^{\otimes n+1}$. As in Section 2, we denote by
$d_i$ the face maps of the simplicial module $C_\bullet(A)$.  First, we recall (see Cuvier \cite[Th\'{e}or\`{e}me III.1]{Cuv}) that $K[U_{n+1}]$ can be made into a presimplicial module, and by abuse of notation, we continue to denote its face maps by $d_i$. Then, the face maps of the presimplicial module $V_\bullet(A)$, still denoted by $d_i$, are determined by (see \cite[Chapitre III]{Cuv})
\begin{equation}\label{eqcuv}
d_i:V_{n}(A)=K[U_{n+1}]\otimes A^{\otimes n+1}\longrightarrow K[U_n]\otimes A^{\otimes n}=V_{n-1}(A) \qquad (\sigma \otimes (a_0,...,a_n))\mapsto d_i\sigma\otimes (\omega(d_i\sigma )\cdot d_i(\omega(\sigma)^{-1}\cdot (a_0,...,a_n)))
\end{equation} The homology of the presimplicial module $V_\bullet(A)$ is typically denoted by $HZ_\bullet(A)$ (see \cite{Cuv1}).  Further, we know  that the canonical inclusions $\iota^A_n: C_n(A)=A^{\otimes n+1}\longrightarrow V_n(A)=K[U_{n+1}]\otimes A^{\otimes n+1}$ induce a quasi-isomorphism of   complexes  and 
$\iota^A_\bullet:C_\bullet(A)\longrightarrow V_\bullet(A)$ has a retraction given by (see \cite[Th\'{e}or\`{e}me III.4]{Cuv1})
\begin{equation}\label{rtr4}
\zeta^A_n:V_n(A)\longrightarrow C_n(A)\qquad \sigma\otimes (a_0,...,a_n)\mapsto sgn(\omega(\sigma))\omega(\sigma)^{-1}\cdot (a_0,...,a_n)
\end{equation} 
 We now have the following result. 

\begin{thm}\label{P5.dc}
 Let $C$ be a cocommutative coalgebra and let $\Phi:C\longrightarrow Hom_K(A,A')$ be a measuring between  algebras. For each  $x\in C$, the maps
 \begin{equation}
 V_n^\Phi(x):V_n(A)\longrightarrow V_n(A') \qquad \sigma\otimes (a_0,...,a_n)\mapsto \sigma\otimes x(a_0,...,a_n)=\sigma\otimes (x_{(1)}(a_0),...,x_{(n+1)}(a_n))
 \end{equation} induce a morphism of presimiplicial modules. Further, the following diagrams commute
 \begin{equation}\label{525cd}
 \begin{array}{lll}
 \begin{CD}
 C_n(A)@>\iota^A_n>> V_n(A)\\
 @VC_n^\Phi(x)VV @VV{V_n^\Phi(x)}V \\
  C_n(A')@>\iota^{A'}_n>> V_n(A')\\
 \end{CD} &\qquad\qquad &\begin{CD}
 V_n(A)@>\zeta^A_n>> C_n(A)\\
 @V{V_n^\Phi(x)}VV @VVC_n^\Phi(x) V \\
  V_n(A')@>\zeta^{A'}_n>> C_n(A')\\
 \end{CD}\\
 \end{array}
 \end{equation}
\end{thm}

\begin{proof} Since $C$ is cocommutative, we note that 
\begin{equation}
x(\tau\cdot  (a_0,...,a_n))= \tau\cdot (x(a_0,...,a_n)) 
\end{equation} for any $x\in C$, $(a_0,...,a_n)\in C_n(A)$ and any permutation $\tau\in S_{n+1}$. We know already that the measurings are well behaved with respect to the face maps of the simplicial modules $C_\bullet(A)$ and $C_\bullet(A')$. From the definitions in \eqref{eqcuv} and \eqref{rtr4}, the result is now clear. 

\end{proof}

The following result now shows that the isomorphisms in \eqref{5.9isob} are compatible with measurings.

\begin{Thm}\label{T5.7v}
Let $C$ be a cocommutative $K$-coalgebra and let $\Phi:C\longrightarrow Hom_K(A,A')$ be a  measuring between  algebras. For each $x\in C$, the following diagram commutes
\begin{equation}\label{5.32cd}
\begin{CD}
Prim(HL_\bullet(gl(A)))@>\cong>> HH_{\bullet-1}(A)\\
@VHL_\bullet^{gl(\Phi)}(x)VV @VVHH_{\bullet-1}^\Phi(x) V \\
Prim(HL_\bullet(gl(A')))@>\cong>> HH_{\bullet-1}(A')\\
\end{CD}
\end{equation} where the horizontal isomorphisms are as given by \eqref{5.9isob}.
\end{Thm}

\begin{proof}
We have noted before that there are  quasi-isomorphisms of complexes   $C_\bullet(A)\longrightarrow V_\bullet(A)$ and
$C_\bullet(A')\longrightarrow V_\bullet(A')$.   From \eqref{step5.1g} and Proposition \ref{P5.dc}, we now have commutative diagrams for each $x\in C$
\begin{equation}\label{step5.1gc}
\begin{array}{lll}
 \begin{CD}
 HH_n(A)@>\cong>> H_n(V_\bullet(A))=HZ_n(A)\\
 @VHH_n^\Phi(x)VV @VV{HZ_n^\Phi(x)}V \\
  HH_n(A')@>\cong>> H_n(V_\bullet(A'))=HZ_n(A')\\
 \end{CD} &\qquad\qquad& 
\begin{CD}
HL_{n+1}(gl(A)) @>\cong>> H_{n}(L_\bullet(A))\\
@VHL_{n+1}^{\Psi}(x)VV @VVH_{n}^{\Phi}(x)V\\
HL_{n+1}(gl(A')) @>\cong>> H_{n}(L_\bullet(A'))\\
\end{CD}
\\ 
\end{array}
\end{equation}  The canonical inclusions $V_n(A)=K[U_{n+1}]\otimes A^{\otimes n+1}\hookrightarrow K[S_{n+1}]\otimes A^{\otimes n+1}=L_n(A)$ are obviously compatible with the maps induced by the measuring  $\Phi$. From \cite[$\S$ III.2]{Cuv}, we know that the maps on $V_\bullet(A)$ induced by restricting to  the subcomplex $V_\bullet(A)=K[U_{\bullet+1}]\otimes A^{\otimes \bullet+1}\subseteq K[S_{\bullet+1}]\otimes A^{\otimes \bullet+1}=L_\bullet(A)$ coincide with the differentials on $V_\bullet(A)$. Since this restriction determines the isomorphisms $HH_{\bullet-1}(A)\cong Prim(HL_\bullet(gl(A)))$ in \eqref{5.9isob}, it follows that the diagram 
\eqref{5.32cd} commutes.
\end{proof} 

\section{Measurings and product structures on Lie homology and Leibniz homology}

We continue with the cocommutative algebra $C$ and the measuring $\Phi:C\longrightarrow Hom_K(A,A')$ between   algebras 
$A$, $A'$. As described in Section 5, this induces a measuring $gl(\Phi):C\longrightarrow Hom_K(gl(A),gl(A'))$ between Lie algebras, as well as between Leibniz algebras.

\smallskip For the Lie algebra $gl(A)$, we consider its Chevalley-Eilenberg complex $CE_\bullet(gl(A))$ with $CE_n(gl(A)):=\wedge^n gl(A)$. For each $r\geq 1$, the Lie algebra $gl_r(K)$ embeds as a Lie subalgebra of $gl_r(A)$ and hence has an adjoint action on $gl_r(A)$ (see \cite[$\S$ 10.2.9]{Lod}). We denote by $\overline{CE}_n(gl_r(A)):=(\wedge^n gl_r(A))_{gl_r(K)}$ the space of coinvariants of this adjoint action.  We know (see \cite[$\S$ 10.2.9]{Lod}) that the differential on $CE_\bullet(gl_r(A))$ descends to $\overline{CE}_\bullet(gl_r(A))$ and that the epimorphisms
\begin{equation}\label{epi6.1}
CE_n(gl_r(A))\longrightarrow \overline{CE}_n(gl_r(A))\qquad n\geq 1
\end{equation} determine a quasi-isomorphism of complexes. We denote by $\overline{CE}_\bullet(gl(A))$ the colimit of the complexes $\overline{CE}_\bullet(gl_r(A))$  as $r\longrightarrow \infty$. 

\begin{lem}
\label{L6.1} The measuring $\Phi:C\longrightarrow Hom_K(A,A')$ induces  morphisms $\overline{CE}_\bullet^{gl(\Phi)}(x):\overline{CE}_\bullet(gl(A))\longrightarrow \overline{CE}_\bullet(gl(A'))$ of complexes for each $x\in C$. Moreover, these morphisms fit into the following commutative diagram
\begin{equation}\label{diag6.1}
\begin{CD}
CE_\bullet(gl(A)) @>>> \overline{CE}_\bullet(gl(A))\\
@VCE_\bullet^{gl(\Phi)}(x)VV @VV\overline{CE}_\bullet^{gl(\Phi)}(x)V \\
CE_\bullet(gl(A')) @>>> \overline{CE}_\bullet(gl(A'))\\
\end{CD}
\end{equation} for each $x\in C$, where the horizontal arrows in \eqref{diag6.1} are as determined by \eqref{epi6.1}.
\end{lem}

\begin{proof} For each $r\geq 1$, the measuring $\Phi$ induces a measuring $gl_r(\Phi):C\longrightarrow Hom_K(gl_r(A),gl_r(A'))$ between Lie algebras. 
We know that $\overline{CE}_n(gl_r(A))$  can be described alternatively as the space of coinvariants (see \cite[$\S$ 10.2.10]{Lod} for details)
\begin{equation}\label{xy6.1}
\overline{CE}_n(gl_r(A)):=(\wedge^n gl_r(A))_{gl_r(K)}=((gl_r(K)^{\otimes n})_{gl_r(K)}\otimes A^{\otimes n})_{S_n}
\end{equation} Here, the action of the symmetric group $S_n$   appearing on the right hand side of \eqref{xy6.1} is by permuting variables and multiplying by the sign, while the action of $gl_r(K)$ on the $A^{\otimes n}$ part  appearing in \eqref{xy6.1} is trivial. It is now clear that the maps $CE_n^{gl_r(\Phi)}(x):CE_n(gl_r(A))\longrightarrow CE_n(gl_r(A'))$ descend to $\overline{CE}_n^{gl_r(\Phi)}(x):\overline{CE}_n(gl_r(A))\longrightarrow \overline{CE}_n(gl_r(A'))$.  We already know that each 
$CE_\bullet^{gl_r(\Phi)}(x):CE_\bullet(gl_r(A))\longrightarrow CE_\bullet(gl_r(A'))$ is a morphism of complexes. Now since the maps $CE_n(gl_r(A))\longrightarrow \overline{CE}_n(gl_r(A))$ in \eqref{epi6.1} are epimorphisms, it follows that the maps $\overline{CE}_n^{gl_r(\Phi)}(x)$ are well behaved with respect to the differentials on $\overline{CE}_\bullet(gl_r(A))$ and $\overline{CE}_\bullet(gl_r(A'))$. Taking colimits as $r\longrightarrow \infty$, it follows that the diagram \eqref{diag6.1} commutes.
\end{proof}

The complex $\overline{CE}_\bullet(gl(A))=(\wedge^\bullet gl(A))_{gl(K)}$ also carries the structure of an algebra (see \cite[$\S$ 10.2.13]{Lod}) with product determined  by maps 
\begin{equation}\label{prod6}\small
\begin{array}{c}
(\wedge^p gl(A))_{gl(K)}\otimes (\wedge^q gl(A))_{gl(K)}\longrightarrow (\wedge^{p+q} (gl(A)\oplus gl(A)))_{gl(K)}
\xrightarrow{\quad\wedge^{p+q}(\_\_\oplus\_\_)_{gl(K)}\quad }  (\wedge^{p+q} gl(A))_{gl(K)}\\
(\alpha_1\wedge ... \wedge \alpha_p)\otimes (\beta_1\wedge ...\wedge \beta_q)\mapsto ((\alpha_1\oplus 0)\wedge ...\wedge (\alpha_p\oplus 0)\wedge 
(0\oplus \beta_1)\wedge ...\wedge (0\oplus \beta_q))=(\alpha_1\wedge ... \wedge \alpha_p\wedge \beta_1\wedge ...\wedge \beta_q)\\
\end{array}
\end{equation} for $p$, $q\geq 0$. The first map in \eqref{prod6} is induced by the canonical morphism $(\wedge^p gl(A))\otimes (\wedge^q gl(A))\longrightarrow (\wedge^{p+q} (gl(A)\oplus gl(A)))$ while the second is induced by the operation $\oplus: gl(A)\times gl(A)\longrightarrow gl(A)$ described in \cite[$\S$ 10.2.12]{Lod}. At the homology level, the product in \eqref{prod6} makes
$H_\bullet(gl(A))=H_\bullet(CE_\bullet(gl(A)))=H_\bullet(\overline{CE}_\bullet(gl(A)))$ into an  algebra. We are now ready to show that the measuring  $\Phi:C\longrightarrow Hom_K(A,A')$ induces a measuring between the algebras $H_\bullet(gl(A))$ and $H_\bullet(gl(A'))$. 

\begin{thm}\label{P6.2}
Let $C$ be a cocommutative coalgebra and let $\Phi:C\longrightarrow Hom_K(A,A')$ be a measuring between  algebras. Then, the following
\begin{equation}\label{mslie6}
H_\bullet^{gl(\Phi)}:C\longrightarrow Hom_K(H_\bullet(gl(A)),H_\bullet(gl(A')))\qquad x\mapsto H_\bullet^{gl(\Phi)}(x):H_\bullet(gl(A))\longrightarrow H_\bullet(gl(A'))
\end{equation} determines a  measuring of algebras from $H_\bullet(gl(A))$ to $H_\bullet(gl(A'))$. 
\end{thm}

\begin{proof}
Let $x\in C$.  From the quasi-isomorphism in \eqref{epi6.1} and the commutativity of the diagram \eqref{diag6.1}, it follows that the maps $H_\bullet^{gl(\Phi)}(x):H_\bullet(gl(A))\longrightarrow H_\bullet(gl(A'))$ can also be induced by the maps $\overline{CE}_\bullet^{gl(\Phi)}(x)$. We consider $\alpha_1\wedge ... \wedge \alpha_p\in \overline{CE}_p(gl(A))$ and $\beta_1\wedge ...\wedge \beta_q\in \overline{CE}_q(gl(A))$. The result is now clear from the fact that we have
\begin{equation}
\begin{array}{ll}
\overline{CE}_{p+q}^{gl(\Phi)}(x)((\alpha_1\wedge ... \wedge \alpha_p)(\beta_1\wedge ...\wedge \beta_q))&=\overline{CE}_{p+q}^{gl(\Phi)}(x)(\alpha_1\wedge ... \wedge \alpha_p\wedge \beta_1\wedge ...\wedge \beta_q )\\
&= x_{(1)}(\alpha_1)\wedge ... \wedge x_{(p)}(\alpha_p)\wedge x_{(p+1)}(\beta_1)\wedge ...\wedge x_{(p+q)}(\beta_q)\\
&=\overline{CE}_{p}^{gl(\Phi)}(x_{(1)})(\alpha_1\wedge ... \wedge \alpha_p)\overline{CE}_{q}^{gl(\Phi)}(x_{(2)})(\beta_1\wedge ...\wedge \beta_q) \\
\end{array}
\end{equation} 
\end{proof}

Our next aim is to show that the measuring in Proposition \ref{P6.2} allows us to define a new enrichment of algebras over the symmetric monoidal category  $coCoalg_K$ of cocommutative coalgebras. For this, we define a new category $\overline{ALG}^c_K$ whose objects are  $K$-algebras, and whose Hom objects are given by 
\begin{equation}\label{enr6tt}
\overline{ALG}^c_K(A,A'):=\mathcal M_c(H_\bullet(gl(A)),H_\bullet(gl(A')))
\end{equation} for $K$-algebras $A$, $A'$. In \eqref{enr6tt}, $\mathcal M(H_\bullet(gl(A)),H_\bullet(gl(A')))$ denotes the universal cocommutative measuring coalgebra from $H_\bullet(gl(A))$ to $H_\bullet(gl(A'))$ (see Section 3 for notation).

\begin{thm}
\label{P6.3cx} Let $A$, $A'$ be  $K$-algebras. Then, there are canonical morphisms 
\begin{equation}
\overline\tau(A,A'): \mathcal M_c(A,A')\longrightarrow \mathcal M_c(H_\bullet(gl(A)),H_\bullet(gl(A')))
\end{equation}
of cocommutative $K$-coalgebras.
\end{thm}

\begin{proof}
For  $K$-algebras $A$, $A'$, we  consider the   measuring 
$
\Psi^c(A,A'):
\mathcal M_c(A,A')\longrightarrow Hom_K(A,A')
$ that is universal among cocommutative measurings from $A$ to $A'$. Since $\mathcal M_c(A,A')$ is cocommutative, it follows from Proposition \ref{P6.2} that we have  an induced measuring
\begin{equation}
H_\bullet^{gl(\Psi^c(A,A'))}: \mathcal M_c(A,A')\longrightarrow Hom_K(H_\bullet(gl(A)),H_\bullet(gl(A')))
\end{equation} between the respective Lie algebra homology rings. The coalgebra morphism $\overline\tau(A,A'): \mathcal M_c(A,A')\longrightarrow \mathcal M_c(H_\bullet(gl(A)),H_\bullet(gl(A')))$ now follows directly from the universal property of $\mathcal M_c(H_\bullet(gl(A)),H_\bullet(gl(A')))$. 
\end{proof}

\begin{Thm}
\label{T3.6cx} There is a $coCoalg_K$-enriched functor $ALG_K^c\longrightarrow \overline{ALG}^c_K$  which is identity on objects and whose mapping on Hom objects is given by \begin{equation}\label{3.21tux}
\overline\tau(A,A'): \mathcal M_c(A,A')\longrightarrow \mathcal M_c(H_\bullet(gl(A)),H_\bullet(gl(A')))
\end{equation} for  $K$-algebras $A$, $A'$. 
\end{Thm}

\begin{proof}
The proof of this result is similar to that of Theorem \ref{T3.6}.
\end{proof}

We now come to the Leibniz homology. As mentioned in Section 5, there is a quasi-isomorphism $CL_\bullet(gl(A))\longrightarrow \overline{CL}_\bullet(gl(A))$ of complexes induced by the canonical epimorphisms $CL_n(gl(A))=(gl(A))^{\otimes n}\longrightarrow (gl(A))^{\otimes n}_{gl(K)}=\overline{CL}_n(gl(A))$ to the space of coinvariants. Given the measuring $\Phi:C\longrightarrow Hom_K(A,A')$, it follows from \eqref{5.2f2} that we have a commutative diagram
\begin{equation}\label{diag6.3}
\begin{CD}
CL_\bullet(gl(A)) @>>> \overline{CL}_\bullet(gl(A))\\
@VCL_\bullet^{gl(\Phi)}(x)VV @VV\overline{CL}_\bullet^{gl(\Phi)}(x)V \\
CL_\bullet(gl(A')) @>>> \overline{CL}_\bullet(gl(A'))\\
\end{CD}
\end{equation}
for each $x\in C$. We recall (see \cite[$\S$ III.2]{Cuv}) that the Leibniz homology $HL_\bullet(gl(A))=H_\bullet(CL_\bullet(gl(A)))=H_\bullet(\overline{CL}_\bullet(gl(A)))$ becomes an algebra with product structure induced by \begin{equation}\label{prod6L}\small 
\begin{array}{c}
(gl(A))^{\otimes p}_{gl(K)}\otimes (gl(A))^{\otimes q}_{gl(K)}\longrightarrow (gl(A)\oplus gl(A))^{p+q}_{gl(K)}
\xrightarrow{\quad\otimes^{p+q}(\_\_\oplus\_\_)_{gl(K)}\quad }  (gl(A))^{p+q}_{gl(K)}\\
(\alpha_1\otimes ... \otimes \alpha_p)\otimes (\beta_1\otimes ... \otimes \beta_q)\mapsto ((\alpha_1\oplus 0)\otimes ...\otimes (\alpha_p\oplus 0)\otimes 
(0\oplus \beta_1)\otimes ...\otimes (0\oplus \beta_q))=(\alpha_1\otimes ... \otimes \alpha_p\otimes \beta_1\otimes...\otimes\beta_q)\\
\end{array}
\end{equation} for $p$, $q\geq 0$. As in \eqref{prod6}, the first map in \eqref{prod6L} is induced by the canonical morphism $(gl(A))^{\otimes p}\otimes (gl(A))^{\otimes q}\longrightarrow (gl(A)\oplus gl(A))^{\otimes p+q}$ while the second is induced by the operation $\oplus: gl(A)\times gl(A)\longrightarrow gl(A)$ described in \cite[$\S$ 10.2.12]{Lod}.   We  will now show that the measuring  $\Phi:C\longrightarrow Hom_K(A,A')$ induces a measuring between the algebras $HL_\bullet(gl(A))$ and $HL_\bullet(gl(A'))$. 

\begin{thm}\label{P6.2L}
Let $C$ be a cocommutative coalgebra and let $\Phi:C\longrightarrow Hom_K(A,A')$ be a measuring between  algebras. Then, the following
\begin{equation}\label{mslie6Lx}
HL_\bullet^{gl(\Phi)}:C\longrightarrow Hom_K(HL_\bullet(gl(A)),HL_\bullet(gl(A')))\qquad x\mapsto HL_\bullet^{gl(\Phi)}(x):HL_\bullet(gl(A))\longrightarrow HL_\bullet(gl(A'))
\end{equation} determines a cocommutative measuring of algebras from $HL_\bullet(gl(A))$ to $HL_\bullet(gl(A'))$. 
\end{thm}

\begin{proof} Considering the quasi-isomorphisms  $CL_\bullet(gl(A))\longrightarrow \overline{CL}_\bullet(gl(A))$ and $CL_\bullet(gl(A'))\longrightarrow \overline{CL}_\bullet(gl(A'))$, as well as the commutativity of the diagram \eqref{diag6.3}, we see that the maps $HL_\bullet^{gl(\Phi)}(x):HL_\bullet(gl(A))\longrightarrow HL_\bullet(gl(A'))$ are also be induced by the maps $\overline{CL}_\bullet^{gl(\Phi)}(x)$ for $x\in C$. 
We let $\alpha_1\otimes ... \otimes\alpha_p\in \overline{CL}_p(gl(A))$ and $\beta_1\otimes ...\otimes \beta_q\in \overline{CL}_q(gl(A))$. 
The result is now clear from the fact that for any $x\in C$, we have
\begin{equation}
\begin{array}{ll}
\overline{CL}_{p+q}^{gl(\Phi)}(x)((\alpha_1\otimes ... \otimes\alpha_p)(\beta_1\otimes ...\otimes\beta_q))&=\overline{CL}_{p+q}^{gl(\Phi)}(x)(\alpha_1\otimes ... \otimes \alpha_p\otimes \beta_1\otimes ...\otimes \beta_q )\\
&= x_{(1)}(\alpha_1)\otimes  ... \otimes x_{(p)}(\alpha_p)\otimes x_{(p+1)}(\beta_1)\otimes ...\otimes x_{(p+q)}(\beta_q)\\
&=\overline{CL}_{p}^{gl(\Phi)}(x_{(1)})(\alpha_1\otimes  ... \otimes \alpha_p)\overline{CL}_{q}^{gl(\Phi)}(x_{(2)})(\beta_1\otimes ...\otimes \beta_q) \\
\end{array}
\end{equation}  
\end{proof}

The last step in this section is to define an enrichment of  algebras over cocommutative coalgebras using the measuring in Proposition \ref{P6.2L}. Similar to \eqref{enr6tt}, we define a new category $\widehat{ALG}^c_K$ whose objects are  $K$-algebras, and whose Hom objects are given by 
\begin{equation}\label{enr6tm}
\widehat{ALG}^c_K(A,A'):=\mathcal M_c(HL_\bullet(gl(A)),HL_\bullet(gl(A')))
\end{equation} for  $K$-algebras $A$, $A'$. Here, $\mathcal M_c(H_\bullet(gl(A)),H_\bullet(gl(A')))$ is the universal cocommutative measuring coalgebra from $HL_\bullet(gl(A))$ to $HL_\bullet(gl(A'))$.

\begin{Thm}
\label{T3.6cxL} Let $A$, $A'$ be  $K$-algebras. Then, there are canonical maps 
\begin{equation}
\widehat\tau(A,A'): \mathcal M_c(A,A')\longrightarrow \mathcal M_c(HL_\bullet(gl(A)),HL_\bullet(gl(A')))
\end{equation}
of cocommutative $K$-coalgebras. Further, there is a $coCoalg_K$-enriched functor $ALG_K^c\longrightarrow \widehat{ALG}^c_K$  which is identity on objects and whose mapping on Hom objects is given by \begin{equation}\label{3.21tuxL}
\widehat\tau(A,A'): \mathcal M_c(A,A')\longrightarrow \mathcal M_c(HL_\bullet(gl(A)),HL_\bullet(gl(A')))
\end{equation} for  algebras $A$, $A'$. 
\end{Thm}

\begin{proof}
The proof of this result is similar to that of Proposition \ref{P6.3cx} and Theorem \ref{T3.6cx}.
\end{proof}

\section{Involutive algebras, dihedral homology and measurings}

We continue with $K$ being a field of characteristic zero. By an involutive $K$-algebra $R$, we will mean  (see \cite[$\S$ 10.5]{Lod}) a unital associative $K$-algebra equipped with a $K$-linear endomorphism $r\mapsto \hat{r}$ which satisfies
  $\hat{\hat{r}}=r$ and $
\hat{r}\hat{s}=\widehat{sr}
$ for $r$, $s\in R$. The element $\hat{r}$ is called the conjugate of $r$. We always  assume that $\hat{1}=1$. For involutive algebras, the role of cyclic homology is often played by dihedral homology (see \cite{Lod87}, \cite{LP88}, \cite[$\S$ 10.5]{Lod}). In this section, we study how measurings induce maps on dihedral homology of involutive algebras. Further, we study the compatibility of maps induced by measurings with the isomorphisms identifying dihedral homology with the primitive part of Lie algebra homology of symplectic matrices and of skew symmetric matrices. We begin with the following definition.

\begin{defn}\label{D7.1}
Let $R$, $R'$ be involutive $K$-algebras. A cocommutative measuring of involutive algebras from $R$ to $R'$ consists of a linear map 
$\Phi:C\longrightarrow Hom_K(R,R')$ such that

\smallskip
(a) $C$ is a cocommutative $K$-coalgebra and $\Phi:C\longrightarrow Hom_K(R,R')$ is a measuring of algebras from $R$ to $R'$.

\smallskip
(b) The measuring $\Phi$ is compatible with conjugation, i.e., we have $\Phi(x)(\hat{r})=\widehat{\Phi(x)(r)}$ for $x\in C$ and $r\in R$.
\end{defn} For $n\geq 0$, we recall that the dihedral group $D_{n+1}$ is generated by symbols
\begin{equation}\label{dihed7}
D_{n+1}=\{\mbox{$u_n$, $v_n$, $\vert$ $u_n^{n+1}=v_n^2=1$ and $v_nu_nv_n^{-1}=u_n^{-1}$} \}
\end{equation} As in previous sections, we set $C_n(R):=R^{\otimes n+1}$ for $n\geq 0$. Then, the dihedral group $D_{n+1}$ acts on $C_n(R)$ as follows (see \cite[$\S$ 10.5.4]{Lod}) 
\begin{equation}\label{dihed7a}
u_n(r_0,r_1,,...,r_n)=(-1)^n(r_n,r_0,...,r_{n-1})\qquad 
v_n(r_0,r_1,...,r_n)=(-1)^{n(n+1)/2}(\hat{r}_0,\hat{r}_n,\hat{r}_{n-1},...,\hat{r}_1)
\end{equation}
We denote by $\mathcal D_n(R):=C_n(R)_{D_{n+1}}$ the module of coinvariants for the action in \eqref{dihed7a} of the dihedral group $D_{n+1}$ on $C_n(R)$. From \cite[$\S$ 5.2.8]{Lod}, we know that the Hochschild differential $b$ descends to the level of coinvariants to give $\bar{b}:\mathcal D_\bullet(R)
\longrightarrow \mathcal D_{\bullet-1}(R)$. The dihedral homology groups
$H\mathcal D_\bullet(R)$ of the involutive algebra $R$ are obtained as the homologies of the complex $(\mathcal D_\bullet(R),\bar{b})$.

\begin{thm}\label{P7.2}
Let $R$, $R'$ be involutive $K$-algebras. Let $\Phi:C\longrightarrow Hom_K(R,R')$ be a cocommutative measuring of involutive algebras from $R$ to $R'$. For each $x\in C$, we have maps
\begin{equation}\label{7.1eq}
\mathcal D_n^\Phi(x):\mathcal D_n(R)\longrightarrow \mathcal D_n(R')\qquad (r_0,r_1,,...,r_n)\mapsto x(r_0,r_1,...,r_n)= (x_{(1)}(r_0),x_{(2)}(r_1),...,x_{(n+1)}(r_n))
\end{equation} which induce a morphism $\mathcal D_\bullet(x): \mathcal D_\bullet(R)\longrightarrow \mathcal D_\bullet(R')$ of complexes. Accordingly, each $x\in C$ induces a morphism $H\mathcal D^\Phi_\bullet(x):H\mathcal D_\bullet(R)\longrightarrow H\mathcal D_\bullet(R')$ on dihedral homologies.
\end{thm}

\begin{proof}
We first need to show that the morphisms $C_n^\Phi(x):C_n(R)\longrightarrow C_n(R')$ as in \eqref{2.4mao} descend to $\mathcal D_n^\Phi(x):\mathcal D_n(R)=C_n(R)_{D_{n+1}}\longrightarrow C_n(R')_{D_{n+1}}=\mathcal D_n(R')$ at the level of the coinvariant modules. Let 
$(r_0,r_1,...,r_n)\in C_n(R)$. Since $C$ is cocommutative, it follows from the properties in Definition \ref{D7.1} that for $u_n$, $v_n\in D_{n+1}$, we have
\begin{equation}
\begin{array}{ll}
x(u_n(r_0,r_1,,...,r_n))&=(-1)^nx(r_n,r_0,...,r_{n-1})\\
&=(-1)^n(x_{(1)}(r_n),x_{(2)}(r_0),...,x_{(n+1)}(r_{n-1}))\\
&=(-1)^n(x_{(n+1)}(r_n),x_{(1)}(r_0),...,x_{(n)}(r_{n-1}))\\
&=u_n(x_{(1)}(r_0),...,x_{(n)}(r_{n-1}),x_{(n+1)}(r_n))=u_n(x(r_0,r_1,...r_n))\\
&\\
x(v_n(r_0,r_1,,...,r_n))&=(-1)^{n(n+1)/2}x(\hat{r}_0,\hat{r}_n,\hat{r}_{n-1},...,\hat{r}_1)\\
&=(-1)^{n(n+1)/2}(x_{(1)}(\hat{r}_0),x_{(2)}(\hat{r}_n),x_{(3)}(\hat{r}_{n-1}),...,x_{(n+1)}(\hat{r}_1))\\
&=(-1)^{n(n+1)/2}(x_{(1)}(\hat{r}_0),x_{(n+1)}(\hat{r}_n),x_{(n)}(\hat{r}_{n-1}),...,x_{(2)}(\hat{r}_1))\\
&=(-1)^{n(n+1)/2}(\widehat{x_{(1)}(r_0)},\widehat{x_{(n+1)}(r_n)},\widehat{x_{(n)}(r_{n-1})},...,\widehat{x_{(2)}(r_1)})\\
&=v_n(x_{(1)}(r_0),x_{(2)}(r_1),...,x_{(n)}(r_{n-1}),x_{(n+1)}(r_n))=v_n(x(r_0,r_1,...r_n))\\
\end{array}
\end{equation} It remains to show that the maps $\mathcal D_\bullet^\Phi(x)$ are well behaved with respect to the differentials 
on $\mathcal D_\bullet(R)$ and $\mathcal D_\bullet(R')$. For that we consider the diagrams 
\begin{equation}\label{7.5cdk}
\begin{array}{ccc}
\begin{CD}
C_n(R) @>b>> C_{n-1}(R)\\
@VC^\Phi_n(x)VV @VVC^\Phi_{n-1}(x)V \\
C_n(R') @>b>> C_{n-1}(R')\\ 
\end{CD} &\qquad\Rightarrow\qquad& \begin{CD}
\mathcal D_n(R) @>\bar{b}>> \mathcal D_{n-1}(R)\\
@V\mathcal D^\Phi_n(x)VV @VV\mathcal D^\Phi_{n-1}(x)V \\
\mathcal D_n(R') @>\bar{b}>> \mathcal D_{n-1}(R')\\ 
\end{CD}\\
\end{array}
\end{equation} From Proposition \ref{P1.1}, we know that the left hand side diagram in \eqref{7.5cdk} commutes. Further, we know that the Hochschild differentials on $C_\bullet(R)$ and $C_\bullet(R')$ descend to the level of the coinvariants  $\mathcal D_\bullet(R)$ and $\mathcal D_\bullet(R')$ respectively. Since the canonical maps $C_\bullet(R)\longrightarrow \mathcal D_\bullet(R)$ and  $C_\bullet(R')\longrightarrow \mathcal D_\bullet(R')$ are epimorphisms, it follows that the right hand side diagram in \eqref{7.5cdk} also commutes. This proves the result. 
\end{proof}

Let $r\geq 1$. If $R$ is an involutive $K$-algebra, then there is an operation $\alpha\mapsto { ^t}\alpha$ on $gl_r(R)$ given by setting ${^t}\alpha_{ij}:=\hat{\alpha}_{ji}$. Then, the Lie algebra $sk_r(R)$ of skew symmetric matrices is given by
\begin{equation}\label{sk7}
sk_r(R):=\{\mbox{$\alpha\in gl_r(R)$ $\vert$ $^t\alpha=-\alpha$}\}
\end{equation} 
For $r\geq 1$, let $J_r$ denote the $2r\times 2r$ matrix given by $J_r:=j^{\oplus r}$  where $j=\begin{pmatrix}0& 1 \\ -1 & 0\\ \end{pmatrix}$ (see \cite[$\S$ 10.5.3]{Lod}). This is used to define an operation  $\alpha\mapsto { ^T}\alpha$ on $gl_{2r}(R)$ given by setting $^T\alpha:=-J_r{^t}\alpha J_r$. Then, the Lie algebra $sp_{2r}(R)$ of symplectic matrices is given by
\begin{equation}\label{sp7}
sp_{2r}(R):=\{\mbox{$\alpha\in gl_{2r}(R)$ $\vert$ $^T\alpha=-\alpha$}\}
\end{equation} Now let $\Phi:C\longrightarrow Hom_K(R,R')$ be a cocommutative measuring of involutive algebras from $R$ to $R'$. We set $\Psi_r:=gl_r(\Phi):C\longrightarrow 
Hom_K(gl_r(R),gl_r(R'))$ and $\Psi:=gl(\Phi):C\longrightarrow Hom_K(gl(R),gl(R'))$.  

\begin{lem}\label{L7.3} 
For each $x\in C$, we have
\begin{equation}\label{7.8tr}
^t\Psi(x)(\alpha)=\Psi(x)({{ ^t}\alpha}) \qquad ^T\Psi(x)(\beta)=\Psi(x)({{ ^T}\beta})
\end{equation} for $\alpha\in gl_r(R)$ and $\beta\in gl_{2r}(R)$.  Further, $\Psi_\bullet:=gl_\bullet(\Phi):C\longrightarrow Hom_K(gl_\bullet(R),gl_\bullet(R'))$ restricts to measurings
\begin{equation}\label{7.9trf}
sk_r(\Phi):C\longrightarrow 
Hom_K(sk_r(R),sk_r(R'))\qquad sp_{2r}(\Phi):C\longrightarrow 
Hom_K(sp_{2r}(R),sp_{2r}(R'))
\end{equation}
of Lie algebras. 
\end{lem}
\begin{proof} From Definition \ref{D7.1}, it is immediate that $^t\Psi(x)(\alpha)=\Psi(x)({{ ^t}\alpha})$. For $x\in C$, we also notice that
\begin{equation*}
\Psi(x)({{ ^T}\beta})=-\Psi(x)(J_r{^t}\beta J_r)=-\Psi(x_{(1)})(J_r)\Psi(x_{(2)})({ ^t}\beta)\Psi(x_{(3)})(J_r)=-\epsilon_C(x_{(1)})\epsilon_C(x_{(3)})J_r\Psi(x_{(2)})({ ^t}\beta)J_r=
-J_r\Psi(x)({^t}\beta)J_r=^T\Psi(x)(\beta)
\end{equation*} This proves \eqref{7.8tr}. From \eqref{7.8tr} and the definitions in \eqref{sk7}, \eqref{sp7}, it is clear that the maps $\Psi_\bullet(x):gl_\bullet(R)\longrightarrow
gl_\bullet(R')$ restrict to the respective Lie subalgebras of skew symmetric matrices, as well as to the respective Lie subalgebras of symplectic matrices. The result is now clear. 

\end{proof}

We write $sk(R)$ (resp. $sp(R)$) for the colimit of the inclusions $sk_r(R)\hookrightarrow sk_{r+1}(R)$ (resp. the inclusions $sp_{2r}(R)\hookrightarrow sp_{2r+2}(R)$) as $r\longrightarrow \infty$. 
By taking colimits as $r\longrightarrow \infty$ in \eqref{7.9trf}, we obtain   $sk(\Phi):C\longrightarrow Hom_K(sk(R),sk(R'))$ and $sp(\Phi):C\longrightarrow Hom_K(sp(R),sp(R'))$. From 
\cite{LP88}, \cite[$\S$ 10.5]{Lod}, we know that the Lie algebra homologies $H_\bullet(sk(R))$ and $H_\bullet(sp(R))$ are actually Hopf algebras. We now have the following result.

\begin{thm}\label{P7.4}
Let $R$, $R'$ be involutive $K$-algebras. Let $\Phi:C\longrightarrow Hom_K(R,R')$ be a cocommutative measuring of involutive algebras from $R$ to $R'$. For each $x\in C$, we have maps
\begin{equation}\label{7.10map}
H_\bullet^{sk(\Phi)}(x):H_\bullet(sk(R))\longrightarrow H_\bullet(sk(R'))\qquad H_\bullet^{sp(\Phi)}(x):H_\bullet(sp(R))\longrightarrow H_\bullet(sp(R'))
\end{equation}
between Lie algebra homologies. Further, the maps in \eqref{7.10map} restrict to 
\begin{equation}\label{7.11map}
H_\bullet^{sk(\Phi)}(x):Prim(H_\bullet(sk(R)))\longrightarrow Prim(H_\bullet(sk(R')))\qquad H_\bullet^{sp(\Phi)}(x):Prim(H_\bullet(sp(R)))\longrightarrow Prim(H_\bullet(sp(R')))
\end{equation}
on the respective primitive parts of the Lie algebra homologies. 
\end{thm}

\begin{proof} The maps in \eqref{7.10map} follow directly from the fact that $sk(\Phi):C\longrightarrow Hom_K(sk(R),sk(R'))$ and $sp(\Phi):C\longrightarrow Hom_K(sp(R),sp(R'))$  are measurings of Lie algebras. Since  $H_\bullet(sk(R))$ and $H_\bullet(sp(R))$ are  Hopf algebras, the last statement is clear from the general result of Proposition \ref{P5.1}. 

\end{proof} 

We now recall (see \cite[$\S$ 10.5.5]{Lod}) that there are isomorphisms
\begin{equation}\label{7.12dx}
Prim(H_\bullet(sk(R)))\overset{\cong}{\longrightarrow} H\mathcal D_{\bullet-1}(R) \qquad Prim(H_\bullet(sp(R)))\overset{\cong}{\longrightarrow} H\mathcal D_{\bullet-1}(R)
\end{equation} Our final result is to show that the isomorphisms in \eqref{7.12dx} are compatible the maps induced by measurings as in Proposition \ref{P7.2} and Proposition \ref{P7.4}. 

\begin{Thm}\label{T57.4}
Let $R$, $R'$ be involutive $K$-algebras. Let $\Phi:C\longrightarrow Hom_K(R,R')$ be a cocommutative measuring of involutive algebras from $R$ to $R'$.  For each $x\in C$, the following diagrams commute
\begin{equation}\label{7.13coma}
\begin{array}{lll}
\begin{CD}
Prim(H_\bullet(sk(R)))@>\cong>> H\mathcal D_{\bullet-1}(R)\\
@VH_\bullet^{sk(\Phi)}(x)VV @VVH\mathcal D_{\bullet-1}^\Phi(x) V \\
Prim(H_\bullet(sk(R')))@>\cong>> H\mathcal D_{\bullet-1}(R')\\
\end{CD} & \qquad &
\begin{CD}
Prim(H_\bullet(sp(R)))@>\cong>> H\mathcal D_{\bullet-1}(R)\\
@VH_\bullet^{sp(\Phi)}(x)VV @VVH\mathcal D_{\bullet-1}^\Phi(x) V \\
Prim(H_\bullet(sp(R')))@>\cong>> H\mathcal D_{\bullet-1}(R')\\
\end{CD} \\
\end{array}
\end{equation} where the isomorphisms in the horizontal arrows are as given by \eqref{7.12dx}.
\end{Thm}
\begin{proof} In the notation of Remark \ref{R3.2}, we recall that $\tilde C_n(R)=R^{\otimes n+1}/(1-t)$. From the expression for the action of the dihedral group $D_{n+1}$ on $C_n(R)=R^{\otimes n+1}$ in \eqref{dihed7a}, it follows that we have canonical epimorphisms $\tilde C_n(R)\twoheadrightarrow \mathcal D_n(R)=C_n(R)_{D_{n+1}}$. 
Combining with the commutative diagram \eqref{comd5.16} in the proof of Theorem \ref{T5.4}, we now have the commutative diagram
\begin{equation}\label{comd57.16}\small
\begin{CD}
CE_{n+1}(sk_r(R))=\wedge^{n+1}sk_r(R) @>>> CE_{n+1}(gl_r(R))=\wedge^{n+1}gl_r(R) @>\theta^A_{n,r}>>  \tilde C_n(M_r(R))@>tr^A_{n,r}>>\tilde C_n(R) @>>> \mathcal D_n(R) \\
@VCE_{n+1}^{sk_r(\Phi)}(x)VV @VCE_{n+1}^{gl_r(\Phi)}(x)VV @VV\tilde C_n^{M_r(\Phi)}(x)V  @VV\tilde C_n^\Phi(x)V  @VV\mathcal D_n^\Phi(x)V  \\
CE_{n+1}(sk_r(R'))=\wedge^{n+1}sk_r(R') @>>> CE_{n+1}(gl_r(R'))=\wedge^{n+1}gl_r(R')  @>\theta^{A'}_{n,r}>>  \tilde C_n(M_r(R'))@>tr_{n,r}^{A'}>> \tilde C_n(R')@>>> \mathcal D_n(R')\\
\end{CD}
\end{equation}  It is clear that the commutativity in \eqref{comd57.16} passes to the colimit as $r\longrightarrow \infty$. At the level of homologies, we therefore have a commutative diagram
\begin{equation}\label{7.15coma} 
\begin{CD}
H_\bullet(sk(R))@>>> H\mathcal D_{\bullet-1}(R)\\
@VH_\bullet^{sk(\Phi)}(x)VV @VVH\mathcal D_{\bullet-1}^\Phi(x) V \\
H_\bullet(sk(R'))@>>> H\mathcal D_{\bullet-1}(R')\\
\end{CD} 
\end{equation} We know (see \cite[$\S$ 10.5.5]{Lod}) that the isomorphism $Prim(H_\bullet(sk(R)))\overset{\cong}{\longrightarrow} H\mathcal D_{\bullet-1}(R)  $ in \eqref{7.12dx} is obtained by restricting the upper horizontal map in \eqref{7.15coma} to the primitive part of $H_\bullet(sk(R))$.  It is now clear that the left hand side diagram in \eqref{7.13coma} commutes. The case of the right hand side diagram in \eqref{7.13coma} is proved similarly. 
\end{proof}

\end{document}